\documentclass[11pt]{article}
\usepackage[margin=1in]{geometry}
\usepackage{amsmath,amssymb,amsthm}
\usepackage{mathtools}
\usepackage{hyperref}

\newtheorem{theorem}{Theorem}[section]
\newtheorem{lemma}[theorem]{Lemma}
\newtheorem{corollary}[theorem]{Corollary}
\newtheorem{proposition}[theorem]{Proposition}
\theoremstyle{definition}
\newtheorem{example}[theorem]{Example}
\theoremstyle{plain}
\theoremstyle{remark}
\newtheorem{remark}[theorem]{Remark}
\theoremstyle{plain}
\newcommand{\Z}{\mathbb Z}

\title{\textbf{Block Sensitivity and Certificate Complexity of Rotation Symmetric Boolean Functions}\thanks{An earlier version of this same paper appeared as arXiv:2610.17100. The present version retains its original quadratic and quartic results and substantially extends them, notably with a general even-degree block-sensitivity theorem.}}
\author{Yuan Li$^{*}$ \and Jing Zhang}
\date{}

\begin{document}
\maketitle

\begin{center}
\small
Yuan Li: Department of Mathematics, Winston--Salem State University,
NC 27110, USA.\\
E-mail: \texttt{liyu@wssu.edu}\\[2pt]
Jing Zhang: Mathematics Department, Governors State University,
IL 60484, USA.\\
E-mail: \texttt{jzhang@govst.edu}\\[4pt]
$^{*}$ Corresponding author.\\[2pt]
\end{center}

\begin{abstract}
We investigate block sensitivity and certificate complexity of monomial
rotation-symmetric Boolean functions, obtaining exact formulas and structural
classifications for several families. For odd-degree monomial rotation-symmetric
functions, the equality $s(f)=n$ established in our earlier work yields
$bs(f)=C(f)=n$. For quadratic monomial rotation-symmetric
functions, we determine both measures completely and prove that block
sensitivity equals certificate complexity. For consecutive quartic monomials,
we determine sensitivity, block sensitivity, and certificate complexity in every
residue class of the number of variables. Our principal result is an exact
global block-sensitivity formula for consecutive monomials of every even degree
$d\ge4$. Writing $n=q(d+2)+\rho$, where $0\le\rho<d+2$ and
$q\ge\max\{2,d/2-1\}$, we prove
\[
bs(F_{d,n})=n-\min\left\{\left\lceil\frac{\rho}{2}\right\rceil,d+2-\rho\right\}.
\]
We also establish the corresponding sensitivity formula by a self-contained
argument, yielding arbitrarily large gaps between sensitivity and block
sensitivity as the degree varies. For certificate complexity, we obtain exact pointwise formulas
for structured run-and-gap configurations and establish additional global
equalities, including $C(F_{6,8q+5})=bs(F_{6,8q+5})=8q+2$ for $q\ge2$.
Our methods combine cyclic run decomposition, parity arguments, explicit
sensitive-block constructions, and constant-subcube certificates. The exact
global certificate complexity for arbitrary even degrees remains open.
\end{abstract}

\section{Introduction}
\par\noindent\textbf{Main contribution.} Theorem~\ref{thm:complete-even-bs}
gives a complete closed formula for global block sensitivity of
consecutive even-degree monomial rotation symmetric Boolean functions
at sufficiently large lengths. Its proof is independent of the
companion paper's universal sensitivity inequality. The quadratic and quartic classifications from the earlier version
remain original contributions of the present paper and are incorporated
as integral parts of the expanded treatment; they are not results
borrowed from independent prior work.

A Boolean function on $n$ variables is a map
\[
f:\{0,1\}^n\longrightarrow\{0,1\}.
\]
Among the fundamental complexity measures associated with Boolean functions
are sensitivity, block sensitivity, and certificate complexity.  Block
sensitivity was introduced by Nisan in connection with decision-tree and
parallel computation complexity \cite{Nisan1991}; the relations among these
measures have since played a central role in Boolean-function complexity,
including the sensitivity conjecture of Nisan and Szegedy
\cite{NisanSzegedy1994}, proved by Huang \cite{Huang2019}.  For every
Boolean function $f$,
\[
s(f)\le bs(f)\le C(f).
\]
Exact formulas for these measures are nevertheless known only for special
classes.  In our earlier joint work, we studied sensitivity and block sensitivity of elementary
symmetric Boolean functions \cite{ZhangLiAdeyeye2021}, and
certificate complexity for this class \cite{ZhangLi2022,ZhangLi2026Cert}.  General
quantitative relations between sensitivity and certificate complexity were
developed by He et al.~\cite{HeEtAl2019}.

Rotation symmetric Boolean functions form an important class of Boolean
functions invariant under cyclic permutations of their variables.  Their
counting and cryptographic properties were studied systematically by
St\u{a}nic\u{a} and Maitra \cite{StanicaMaitra2008}; subsequent work on
monomial rotation symmetric functions includes affine-equivalence and weight
questions for quartic functions \cite{CusickCheon2014}, and detailed work on
quadratic MRS functions \cite{ChirvasituCusick2024}.  Higher-order derivative
structure, linear structures, and fast points of RSBFs have also been studied
by Sun and Shi \cite{SunShi2024}.  From the complexity
point of view, cyclically invariant and more generally transitive Boolean
functions have been studied by Chakraborty \cite{Chakraborty2011} and, more
recently, by Chaubal and G\'al \cite{ChaubalGal2023}.  Recent work of
Chakraborty and G\'al also studies the number and structure of minimal
sensitive blocks \cite{ChakrabortyGal2026}.  These works concern general
bounds, constructions, or sensitive-block parameters rather than the exact
certificate- and block-sensitivity formulas for the monomial rotation
symmetric families considered here.

\par\smallskip\noindent\textbf{Relation to the earlier version of this paper.}\par\noindent An earlier version of this article (arXiv:2610.17100) presented, among other results, our exact quadratic and consecutive-quartic classifications. Those are original results of the present authors and remain contributions of this expanded article. The present version retains these results and adds the general even-degree theorem and further certificate-complexity results. The earlier posting documents the authors' own discoveries; it is not independent prior work.\par

Our companion manuscript, \emph{Sensitivity of Rotation Symmetric Boolean Functions via Binary Necklaces}~\cite{LiZhangRSBFSensitivity}, develops the sensitivity theory for these families. For the consecutive even-degree family in the sufficiently-large-length range, we give an independent sensitivity proof in Corollary~\ref{cor:intrinsic-sensitivity}. The purpose of
the present paper is to study block sensitivity and certificate complexity for
the same class of functions, and in particular to compare these measures with
the known sensitivity values. The quartic sensitivity values are reproved for the sufficiently-large-length range by Corollary~\ref{cor:intrinsic-sensitivity}; small lengths are treated separately. The companion manuscript is cited for background, not as a necessary input to our general even-degree sensitivity proof. The proof of our $n\equiv4\pmod6$ certificate bound is supplied independently in the present paper, with the finite-state identities and their exact verifier reproduced in Appendix~\ref{app:finite-state}; it does not rely on the finite-state excursion argument from an earlier companion draft. We first formulate certificate complexity and
block sensitivity on the necklace space: the sensitive-block hypergraph at a
necklace is decomposed into cyclic difference families indexed by
opposite-value necklaces, so that certificate complexity is its transversal
number and block sensitivity is its matching number.  We also determine the
stabilizer-orbit structure of these difference families.  To the best of our
knowledge, these exact block-sensitivity and certificate-complexity formulas
were not established in independent earlier work; the quadratic and quartic
formulas were already reported in the authors' earlier version of this paper.

The strict separation in this paper also distinguishes cyclic symmetry from
full permutation symmetry.  In the fully symmetric setting, it is known that
for every nonconstant fully symmetric Boolean function,
$bs(f)\le\max\{s(f),C(f)-1\}$~\cite{LiZhangSymmetric2026}.
In particular such a function cannot satisfy $s(f)<bs(f)=C(f)$.  By contrast, the quadratic and consecutive
quartic MRS families below provide infinite families with exactly this
pattern.  Thus cyclic invariance alone permits a separation forbidden by
full symmetry.

Our first observation completely settles the odd-degree case. Every
odd-degree monomial rotation symmetric Boolean function has sensitivity
$n$. Since
\[
s(f)\le C(f)\le n,
\]
its certificate complexity is necessarily $n$.

The quadratic case is considerably more interesting. Let $Q_{n,k}$ denote
the rotation symmetric Boolean function generated by $x_0x_k$. We obtain an
exact pointwise description of $C(Q_{n,k},x)$ using an associated graph
consisting of disjoint cycles. This reduces the certificate problem to an
independence-number problem on certain induced subgraphs of cycles.

If
\[
d=\gcd(n,k),\qquad m=\frac nd,
\]
then, apart from the short-orbit case $k=n/2$,
\[
C(Q_{n,k})=
\begin{cases}
n,&4\mid m,\\
n-d,&4\nmid m.
\end{cases}
\]
In the short-orbit case, $C(Q_{n,n/2})=n$.

Comparison with the corresponding sensitivity formula gives
\[
s(Q_{n,k})<C(Q_{n,k})
\]
if and only if $k\ne n/2$ and
\[
m\equiv2\pmod4.
\]
In that case,
\[
C(Q_{n,k})-s(Q_{n,k})=d.
\]
We further prove that
\[
bs(Q_{n,k})=C(Q_{n,k})
\]
for every quadratic MRS function.  Thus the same congruence condition
\(m\equiv2\pmod4\) characterizes precisely when
\[
s(Q_{n,k})<bs(Q_{n,k})=C(Q_{n,k}).
\]

\section{Preliminaries}

Throughout, coordinates are indexed by
\[
\Z_n=\{0,1,\ldots,n-1\},
\]
with all subscripts interpreted modulo $n$.  A \emph{cyclic interval of length $\ell$} is a set of the form
$\{a,a+1,\ldots,a+\ell-1\}\subseteq\Z_n$, understood modulo $n$.
For $t\in\Z_n$ and $x\in\{0,1\}^n$, define the cyclic rotation
$R_t$ by $(R_tx)_i=x_{i-t}$, with subscripts modulo $n$. A Boolean
function $f$ is \emph{rotation symmetric} if $f(R_tx)=f(x)$ for every
$x$ and $t$. A \emph{binary necklace} $[x]=\{R_tx:t\in\Z_n\}$
is the orbit of a binary word $x$ under cyclic rotation; a statement
about a necklace input means the corresponding rotation-invariant
statement for any representative.
A \emph{maximal one-run} (respectively, \emph{zero-gap}) is a maximal
cyclic string of consecutive ones (respectively, zeros). Unless a
constant word is explicitly considered, the cyclic run decomposition
alternates nonempty one-runs and zero-gaps. We write $(u)_+=\max\{u,0\}$. For binary words, $1^a$ and
$0^b$ denote runs of $a$ ones and $b$ zeros, respectively;
juxtaposition denotes concatenation, and $w^q$ denotes $q$
concatenated copies of the word $w$ (not exponentiation in a field).
The symbol $1^n$ denotes the all-one word. For $i,j\in\Z_n$,
their \emph{cyclic distance} is
$\operatorname{dist}_n(i,j)=\min\{|i-j|,n-|i-j|\}$,
where $i,j$ are represented by integers in $\{0,\ldots,n-1\}$.
All polynomial formulas for Boolean functions are evaluated in
$\mathbb F_2$; $\oplus$ denotes addition modulo two (XOR),
while juxtaposition of bits denotes their product.

For $x\in\{0,1\}^n$ and $S\subseteq\Z_n$, write $x^S$ for the input
obtained from $x$ by flipping all coordinates in $S$. When $S=\{i\}$, write
$x^{(i)}$. Let $e_i\in\{0,1\}^n$ be the vector having a one
only in coordinate $i$, and let $\mathbf1_U$ be the indicator vector
of a set $U\subseteq\Z_n$. Thus $x^U=x\oplus\mathbf1_U$.
For a Boolean function $f$, its first and mixed second Boolean
derivatives are
\[
D_if(x)=f(x^{(i)})\oplus f(x),\qquad
D_iD_jf(x)=\bigoplus_{U\subseteq\{i,j\}} f(x^U)
\quad(i\ne j).
\]
In particular, coordinate $i$ is sensitive at $x$ precisely when
$D_if(x)=1$.

The sensitivity of $f$ at $x$ is
\[
s(f,x)=|\{i\in\Z_n:f(x)\ne f(x^{(i)})\}|,
\qquad
s(f)=\max_x s(f,x).
\]

A certificate for $f$ at $x$ is a set $S\subseteq\Z_n$ such that every
$y\in\{0,1\}^n$ satisfying $y_i=x_i$ for all $i\in S$ also satisfies
$f(y)=f(x)$. The certificate complexity at $x$ is
\[
C(f,x)=\min\{|S|:S\text{ is a certificate for }f\text{ at }x\},
\]
and
\[
C(f)=\max_x C(f,x).
\]
For a nonempty set $B\subseteq\Z_n$, write $x^B$ for the input obtained
by flipping all coordinates in $B$. A block $B$ is \emph{sensitive} for $f$
at $x$ if
\[
f(x)\ne f(x^B).
\]
The block sensitivity $bs(f,x)$ is the maximum number of pairwise disjoint
sensitive blocks at $x$, and
\[
bs(f)=\max_x bs(f,x).
\]
For a necklace $X=[x]$ and a rotation symmetric function $f$, set
$ s(f,X)=s(f,x)$, $bs(f,X)=bs(f,x)$, and $C(f,X)=C(f,x)$.
These definitions do not depend on the representative $x$, since
cyclic rotations permute the coordinates, sensitive blocks, and
certificates bijectively.

It is standard that
\[
s(f,x)\le bs(f,x)\le C(f,x)
\]
for every input $x$, and consequently
\[
s(f)\le bs(f)\le C(f)\le n.
\]

For completeness, we regard a hypergraph on $\Z_n$ as a family of
subsets of $\Z_n$ (its hyperedges). A \emph{transversal} or \emph{hitting
set} intersects every hyperedge, and a \emph{matching} is a subfamily
of pairwise disjoint hyperedges. The minimum transversal size and
maximum matching size are denoted by $\tau$ and $\nu$, respectively.
Both values are zero for the empty hyperedge family. In the sensitive-block
hypergraphs below, every hyperedge is nonempty.

The following standard hitting-set interpretation of certificate complexity
will also be useful. For $x\in\{0,1\}^n$, let
\[
\mathcal B_f(x)=
\{B\subseteq\Z_n:f(x^B)\ne f(x)\}.
\]
Thus $\mathcal B_f(x)$ is the family of sensitive blocks at $x$.

\begin{lemma}[Standard certificate--transversal characterization]
For every Boolean function $f$ and every input $x$,
\[
C(f,x)=\tau(\mathcal B_f(x)),
\]
where $\tau$ denotes the minimum size of a set meeting every member of
$\mathcal B_f(x)$.  Equivalently, it is enough to hit the
inclusion-minimal members of $\mathcal B_f(x)$.
\end{lemma}
\noindent\emph{Reference.} This standard identity follows directly from the definitions of certificates and sensitive blocks; for general background on certificate complexity, see \cite{NisanSzegedy1994,BenDavidHatamiTal2017}. Neither paper is asserted here to be the original source of this exact formulation.

\begin{proof}
If $S$ is a certificate at $x$, then $S$ must meet every
$B\in\mathcal B_f(x)$; otherwise $x^B$ agrees with $x$ on all coordinates
of $S$ but satisfies $f(x^B)\ne f(x)$.

Conversely, suppose $S$ meets every member of $\mathcal B_f(x)$.  If $S$
were not a certificate, there would be an input $y$ agreeing with $x$ on
$S$ but satisfying $f(y)\ne f(x)$.  Let
\[
B=\{i\in\Z_n:x_i\ne y_i\}.
\]
Then $y=x^B$, so $B\in\mathcal B_f(x)$, while $B\cap S=\varnothing$, a
contradiction.  Hence $S$ is a certificate.  Finally, every sensitive block
contains an inclusion-minimal sensitive block, so the transversal number is
unchanged if only the inclusion-minimal members are retained.
\end{proof}

We use this known general viewpoint below as background for rotation-symmetric
refinements; the hitting-set characterization itself is not claimed as new.

For rotation symmetric functions, the same hypergraph admits a natural
decomposition by necklaces.  Let \(\mathcal N_n\) denote the set of binary
necklaces of length \(n\).  If \(X=[x]\) and \(Y=[y]\) are necklaces, define
\[
\Delta_k(x,y)=
\{i\in\Z_n:x_i\ne(R_ky)_i\},
\qquad k\in\Z_n,
\]
and let
\[
\mathcal D(X,Y)=
\{\Delta_k(x,y):k\in\Z_n\}.
\]
Changing the representative \(y\) only reindexes \(k\). Changing the
representative of \(X\) rotates every member of the family by a common cyclic
translation. Thus \(\mathcal D(X,Y)\) is well defined up to cyclic
translation, and in particular its transversal and matching numbers are
independent of the chosen representatives.

For a rotation symmetric Boolean function \(f\), write
\(\bar f(X)=f(x)\), and define the opposite-necklace difference hypergraph
\[
\mathcal H_f(X)=
\bigcup_{\substack{Y\in\mathcal N_n\\
\bar f(Y)\ne\bar f(X)}}\mathcal D(X,Y).
\]

The \emph{matching number} $\nu(\mathcal H)$ of a hypergraph
$\mathcal H$ is the maximum cardinality of a family of pairwise
disjoint hyperedges of $\mathcal H$. By the standard definition of block
sensitivity~\cite{Nisan1991}, $bs(f,x)=\nu(\mathcal B_f(x))$.
Thus the next statement is a convenient necklace reformulation of standard
sensitive-block and certificate identities, not a new general complexity
characterization.

\begin{theorem}[Necklace reformulation of the sensitive-block hypergraph]
Let \(f\) be a rotation symmetric Boolean function and \(X=[x]\).  Then
\[
\boxed{\mathcal H_f([x])=\mathcal B_f(x)\quad\text{for the chosen representative }x.}
\]
Consequently,
\[
\boxed{C(f,X)=\tau(\mathcal H_f(X))}
\]
and
\[
\boxed{bs(f,X)=\nu(\mathcal H_f(X)),}
\]
where \(\nu\) denotes the matching number defined above.  In particular,
\[
\boxed{
C(f)=\max_{X\in\mathcal N_n}\tau(\mathcal H_f(X)),
\qquad
bs(f)=\max_{X\in\mathcal N_n}\nu(\mathcal H_f(X)).
}
\]
\end{theorem}
\noindent\emph{References and scope.} The matching interpretation of block sensitivity is immediate from Nisan's definition \cite{Nisan1991}; the transversal interpretation of certificates is the standard fact in Lemma~2.1. Their necklace-indexed restatement is included here for use in subsequent arguments.

\begin{proof}
Fix \(X=[x]\).  If
\[
B=\Delta_k(x,y)\in\mathcal D(X,Y)
\]
for a necklace \(Y=[y]\) with \(\bar f(Y)\ne\bar f(X)\), then
\[
x^B=R_ky.
\]
Rotation symmetry gives
\[
f(x^B)=f(R_ky)=f(y)=\bar f(Y)\ne\bar f(X)=f(x).
\]
Hence \(B\in\mathcal B_f(x)\), and therefore
\[
\mathcal H_f(X)\subseteq\mathcal B_f(x).
\]

Conversely, let \(B\in\mathcal B_f(x)\) and put \(y=x^B\).  Then
\(f(y)\ne f(x)\).  If \(Y=[y]\), rotation symmetry implies
\[
\bar f(Y)\ne\bar f(X),
\]
and
\[
B=\{i:x_i\ne y_i\}=\Delta_0(x,y).
\]
Thus \(B\in\mathcal D(X,Y)\subseteq\mathcal H_f(X)\).  Hence
\[
\mathcal B_f(x)\subseteq\mathcal H_f(X),
\]
proving equality.

The certificate formula follows from the preceding transversal
characterization.  The block-sensitivity formula follows because a family
of pairwise disjoint sensitive blocks is exactly a matching in
\(\mathcal B_f(x)=\mathcal H_f(X)\).  Finally, both quantities are invariant
under rotation of \(x\), so maximizing over inputs is equivalent to
maximizing over necklaces.
\end{proof}

\begin{remark}
The transversal/matching interpretation of sensitive blocks is standard for
general Boolean functions~\cite{Nisan1991,BenDavidHatamiTal2017}.
The theorem above records the additional
rotation-symmetric structure by decomposing the sensitive-block hypergraph
into cyclic difference families indexed by opposite-value necklaces.
\end{remark}

We next describe the internal orbit structure of a single difference
family.  For a word \(z\in\{0,1\}^n\), write
\[
H_z=\{a\in\Z_n:R_az=z\}
\]
for its rotational stabilizer, and put \(h_z=|H_z|\).

\begin{theorem}[Difference-orbit structure]
Let \(X=[x]\) and \(Y=[y]\) be binary necklaces of length \(n\).  Then
\[
\Delta_k(x,y)=\Delta_\ell(x,y)
\quad\Longleftrightarrow\quad
\ell-k\in H_y.
\]
Consequently,
\[
\boxed{|\mathcal D(X,Y)|=\frac{n}{h_y}=|Y|.}
\]

Moreover, \(H_x\) acts on \(\mathcal D(X,Y)\) by cyclic translation of
blocks.  Every \(H_x\)-orbit has size
\[
\boxed{\frac{h_x}{|H_x\cap H_y|},}
\]
and the number of such orbits is
\[
\boxed{
\frac{n|H_x\cap H_y|}{h_xh_y}
=
\frac{n}{|H_x+H_y|}.
}
\]
Since \(\Z_n\) is cyclic,
\[
|H_x\cap H_y|=\gcd(h_x,h_y),
\qquad
|H_x+H_y|=\operatorname{lcm}(h_x,h_y),
\]
so equivalently
\[
\boxed{
\text{orbit size}=
\frac{h_x}{\gcd(h_x,h_y)},
\qquad
\text{number of orbits}=
\frac{n}{\operatorname{lcm}(h_x,h_y)}.
}
\]
\end{theorem}

\begin{proof}
By definition,
\[
\Delta_k(x,y)=\{i:x_i\ne(R_ky)_i\}.
\]
Thus
\[
\Delta_k(x,y)=\Delta_\ell(x,y)
\]
if and only if
\[
x\oplus R_ky=x\oplus R_\ell y,
\]
which is equivalent to
\[
R_ky=R_\ell y.
\]
Hence
\[
R_{\ell-k}y=y,
\]
or equivalently \(\ell-k\in H_y\).  Therefore the distinct members of
\(\mathcal D(X,Y)\) are naturally indexed by the cosets of \(H_y\) in
\(\Z_n\), giving
\[
|\mathcal D(X,Y)|=\frac{n}{h_y}.
\]

Now let \(a\in H_x\).  Since \(R_ax=x\), rotation of a difference block by
\(a\) gives
\[
\Delta_k(x,y)+a=\Delta_{k+a}(x,y)
\]
with the chosen rotation convention.  Thus \(H_x\) acts on
\(\mathcal D(X,Y)\).

The stabilizer in \(H_x\) of the block \(\Delta_k(x,y)\) consists of those
\(a\in H_x\) for which
\[
\Delta_{k+a}(x,y)=\Delta_k(x,y).
\]
By the first part, this is exactly \(H_x\cap H_y\).  The orbit-stabilizer
theorem therefore gives orbit size
\[
\frac{h_x}{|H_x\cap H_y|}.
\]
Dividing the total number \(n/h_y\) of distinct blocks by this orbit size
gives
\[
\frac{n|H_x\cap H_y|}{h_xh_y}.
\]
For finite subgroups,
\[
|H_x+H_y|
=
\frac{|H_x||H_y|}{|H_x\cap H_y|},
\]
which yields the second displayed form.  Finally, subgroups of the cyclic
group \(\Z_n\) are uniquely determined by their orders, so their
intersection and generated subgroup have orders
\(\gcd(h_x,h_y)\) and \(\operatorname{lcm}(h_x,h_y)\), respectively.
\end{proof}

\begin{remark}
Thus the opposite-necklace formula may be organized orbit by orbit.  For a
fixed pair \(X,Y\), the family \(\mathcal D(X,Y)\) has exactly
\(n/h_y\) distinct blocks, arranged into
\[
\frac{n}{\operatorname{lcm}(h_x,h_y)}
\]
orbits under the stabilizer of \(X\).  This gives a structural compression
of the sensitive-block hypergraph whenever \(X\) has nontrivial rotational
symmetry.
\end{remark}

The inequalities \(s\le bs\le C\) allow the certificate framework above to
determine block sensitivity in many cases without a separate optimization.

\begin{corollary}[Sandwich criterion]
Let \(f\) be rotation symmetric and \(X=[x]\).  If
\[
s(f,X)=C(f,X),
\]
then
\[
\boxed{s(f,X)=bs(f,X)=C(f,X).}
\]
At the global level, if
\[
s(f)=C(f),
\]
then
\[
\boxed{s(f)=bs(f)=C(f).}
\]
\end{corollary}
\noindent\emph{Reference.} This is an immediate consequence of the standard inequalities $s(f,x)\le bs(f,x)\le C(f,x)$ and their global versions; see \cite{Nisan1991,NisanSzegedy1994} for background.

\begin{proof}
Both statements follow immediately from
\[
s(f,X)\le bs(f,X)\le C(f,X)
\]
and
\[
s(f)\le bs(f)\le C(f),
\]
respectively.
\end{proof}

For a fixed input \(x\), let
\[
S_f(x)=\{i\in\Z_n:f(x^{(i)})\ne f(x)\}
\]
be the sensitive-coordinate set, and put
\[
N_f(x)=\Z_n\setminus S_f(x).
\]

\begin{theorem}[Nonsensitive-block criterion]
Let \(f\) be a rotation symmetric Boolean function and \(X=[x]\).  Put
\[
r=C(f,X)-s(f,X).
\]
Suppose there exist \(r\) pairwise disjoint sensitive blocks
\[
B_1,\ldots,B_r\subseteq N_f(x).
\]
Then
\[
\boxed{bs(f,X)=C(f,X).}
\]
Equivalently, it is enough to find \(r\) pairwise disjoint difference
blocks
\[
B_j=\Delta_{k_j}(x,y_j)
\]
with
\[
\bar f([y_j])\ne\bar f(X)
\qquad\text{and}\qquad
B_j\subseteq N_f(x).
\]
\end{theorem}

\begin{proof}
The \(s(f,X)\) sensitive coordinates give \(s(f,X)\) pairwise disjoint
sensitive singleton blocks
\[
\{i\},\qquad i\in S_f(x).
\]
By assumption, the blocks \(B_1,\ldots,B_r\) are pairwise disjoint and lie
entirely in \(N_f(x)\), so they are also disjoint from all of these
sensitive singleton blocks.  Hence
\[
bs(f,X)\ge s(f,X)+r=C(f,X).
\]
Since always \(bs(f,X)\le C(f,X)\), equality follows.

The equivalent necklace formulation follows from the
necklace reformulation theorem, which identifies the sensitive blocks
at \(x\) with the members of \(\mathcal H_f(X)\).
\end{proof}

\begin{corollary}
If
\[
C(f,X)=s(f,X)+1
\]
and the nonsensitive-coordinate set \(N_f(x)\) contains a sensitive block,
then
\[
\boxed{bs(f,X)=C(f,X)=s(f,X)+1.}
\]
\end{corollary}

\begin{remark}
The criterion explains the mechanism in several exact families below.  When
\(s<C\), the sensitive singleton blocks already account for \(s\) disjoint
blocks; the remaining task is to pack enough additional sensitive blocks
entirely inside the nonsensitive coordinates.  In the quadratic and
consecutive quartic families, this is precisely how the strict cases
\(s<bs=C\) arise.
\end{remark}

\subsection*{Two examples}

The following small examples illustrate the two roles of the general
framework.

\begin{example}[An opposite-necklace difference family]
Consider
\[
Q_{6,1}(x)=\bigoplus_{i\in\Z_6}x_ix_{i+1},
\]
and let
\[
X=[000000],\qquad Y=[000011].
\]
We have
\[
Q_{6,1}(000000)=0,\qquad Q_{6,1}(000011)=1,
\]
so \(Y\) is an opposite-value necklace for \(X\).  Taking \(x=000000\)
and \(y=000011\), the six difference blocks are
\[
\mathcal D(X,Y)=
\bigl\{
\{4,5\},\{3,4\},\{2,3\},\{1,2\},\{0,1\},\{0,5\}
\bigr\}.
\]
Thus this single opposite necklace contributes the six edges of the cycle
\(C_6\) to the sensitive-block hypergraph at \(X\).

Here \(H_x=\Z_6\), while \(y\) has trivial stabilizer.  The
difference-orbit theorem therefore predicts six distinct blocks forming one
\(H_x\)-orbit of size six, exactly as displayed above.  This example also
shows concretely how an opposite necklace contributes an entire cyclic
family of sensitive blocks rather than a single block.
\end{example}

\begin{example}[The nonsensitive-block criterion for \(Q_{6,1}\)]
Again consider \(Q_{6,1}\), now at
\[
x=000011.
\]
The Boolean derivative is
\[
D_jQ_{6,1}(x)=x_{j-1}\oplus x_{j+1}.
\]
At this input the sensitive coordinates are
\[
S_f(x)=\{0,3,4,5\},
\]
while
\[
N_f(x)=\{1,2\}.
\]
Thus
\[
s(Q_{6,1},x)=4.
\]
The four sensitive coordinates give four pairwise disjoint sensitive
singleton blocks.  Although coordinates \(1\) and \(2\) are individually
nonsensitive, flipping them simultaneously changes the value:
\[
Q_{6,1}(000011)=1,
\qquad
Q_{6,1}(011011)=0.
\]
Hence
\[
B=\{1,2\}\subseteq N_f(x)
\]
is an additional sensitive block, disjoint from all four sensitive
singletons.  Therefore
\[
bs(Q_{6,1},x)\ge5.
\]
The quadratic certificate theorem below gives
\[
C(Q_{6,1})=5,
\]
so
\[
\boxed{
s(Q_{6,1})=4<bs(Q_{6,1})=C(Q_{6,1})=5.
}
\]
This is the smallest quadratic illustration of the nonsensitive-block
criterion.
\end{example}

For $A\subseteq\Z_n$, let
\[
O(A)=\{A+t:t\in\Z_n\}
\]
denote its orbit under cyclic translation. The monomial rotation symmetric
Boolean function generated by $A$ is
\[
M_A(x)=\bigoplus_{B\in O(A)}\prod_{j\in B}x_j,
\]
where each distinct translated monomial is included exactly once.

\section{Odd-Degree Monomial Rotation Symmetric Functions}

\begin{theorem}[Odd-degree MRS functions]
Let $M_A$ be a monomial rotation symmetric Boolean function on $n$
variables. If $|A|$ is odd, then
\[
\boxed{s(M_A)=bs(M_A)=C(M_A)=n.}
\]
\end{theorem}
\begin{proof}
Our earlier sensitivity result \cite{LiZhangRSBFSensitivity} gives
$s(M_A)=n$ for every odd-degree monomial rotation symmetric function.
The standard inequalities
\[
s(M_A)\le bs(M_A)\le C(M_A)\le n
\]
then force equality throughout.
\end{proof}

Thus all three measures are completely determined in odd degree; the
nontrivial behavior begins in even degree.

\section{Quadratic Monomial Rotation Symmetric Functions}

For
\[
1\le k\le\left\lfloor\frac n2\right\rfloor,
\]
let $Q_{n,k}$ denote the quadratic rotation symmetric Boolean function
generated by $x_0x_k$. Except when $n$ is even and $k=n/2$,
\[
Q_{n,k}(x)=\bigoplus_{i\in\Z_n}x_ix_{i+k}.
\]

Define the undirected graph $G_{n,k}=(\Z_n,E_{n,k})$ with vertex set
$\Z_n$ and edge set
\[
E_{n,k}=\{\{i,i+k\}:i\in\Z_n\}.
\]
If $k\ne n/2$, then $G_{n,k}$ is the disjoint union of
$d=\gcd(n,k)$ cycles, each of length $m=n/d$.
Indeed, starting at a vertex $r$, successive additions of $k$ modulo $n$
visit exactly the vertices $r,r+k,r+2k,\ldots$ modulo $n$ in its
connected component. The first return to $r$ occurs after $n/d$ steps;
there are $d$ such components, with no edges between them.
For any graph $G=(V,E)$ and subset $U\subseteq V$, we write $G[U]$
for the subgraph of $G$ induced by $U$, with vertex set $U$ and all
edges of $G$ having both endpoints in $U$.

For $x\in\{0,1\}^n$, define
\[
Z(x)=\{j\in\Z_n:x_{j-k}=x_{j+k}\}.
\]

\begin{lemma}
Assume $k\ne n/2$. For every $x\in\{0,1\}^n$,
\[
\boxed{C(Q_{n,k},x)=n-\alpha(G_{n,k}[Z(x)]),}
\]
where $\alpha(H)$ denotes the independence number of $H$.
\end{lemma}

\begin{proof}
Let $U$ be the set of coordinates left unspecified by a certificate at
$x$. If two adjacent vertices $j$ and $j+k$ both belong to $U$, then after
all coordinates outside $U$ are fixed, the algebraic normal form still
contains the quadratic term $x_jx_{j+k}$. No other term has the same pair
of free variables. Thus the restricted function cannot be constant.
Therefore $U$ must be independent in $G_{n,k}$.

If $j\in U$, then both $j-k$ and $j+k$ are fixed. The two terms involving
$x_j$ have sum
\[
x_{j-k}x_j\oplus x_jx_{j+k}
=x_j(x_{j-k}\oplus x_{j+k}).
\]
Thus the restricted function is independent of $x_j$ exactly when
$x_{j-k}=x_{j+k}$, so $j\in Z(x)$.

Conversely, if $U$ is an independent subset of $Z(x)$, no quadratic term
contains two free coordinates, and the coefficient of every free variable
is zero. Hence the restricted function is constant. Therefore the maximum
number of coordinates that may be left unspecified is
$\alpha(G_{n,k}[Z(x)])$, proving the formula.
\end{proof}

\begin{theorem}
Let $1\le k\le\lfloor n/2\rfloor$. If $n$ is even and $k=n/2$, then
\[
\boxed{C(Q_{n,k})=n.}
\]
Otherwise, put
\[
d=\gcd(n,k),\qquad m=\frac nd.
\]
Then
\[
\boxed{
C(Q_{n,k})=
\begin{cases}
n,&m\equiv0\pmod4,\\[1mm]
n-d,&m\not\equiv0\pmod4.
\end{cases}}
\]
\end{theorem}

\begin{proof}
In the short-orbit case,
\[
Q_{n,n/2}(x)=\bigoplus_{i=0}^{n/2-1}x_ix_{i+n/2}.
\]
At $1^n$ every coordinate is sensitive, so
$C(Q_{n,n/2})=n$.

Now suppose $k\ne n/2$. The graph $G_{n,k}$ consists of $d$ disjoint
cycles of length $m$. On one component write the input values in cyclic
order as
\[
a_0,a_1,\ldots,a_{m-1}
\]
and define
\[
y_i=a_i\oplus a_{i+1},\qquad i\pmod m.
\]
Since the sequence is cyclic,
\[
\bigoplus_{i=0}^{m-1}y_i=0. \tag{4.1}
\]
Moreover,
\[
a_{i-1}=a_{i+1}
\quad\Longleftrightarrow\quad
y_{i-1}=y_i. \tag{4.2}
\]

Suppose first that $4\mid m$. Choose
\[
y=(01)^{m/2}.
\]
It has $m/2$ ones, an even number, so (4.1) holds. No two cyclically
consecutive entries are equal. Thus $Z(x)$ is empty on this component.
Making this choice on every component gives $Z(x)=\varnothing$, and the
lemma yields $C(Q_{n,k},x)=n$. Hence $C(Q_{n,k})=n$.

Now assume $4\nmid m$. If a component contained no vertex of $Z(x)$, then
by (4.2) its transition sequence would alternate cyclically. Hence $m$
would be even. Such an alternating word has $m/2$ ones, and (4.1) would
force $m/2$ to be even. Thus $4\mid m$, a contradiction. Hence every
component contains at least one vertex of $Z(x)$. Choosing one from each
component gives an independent set, so
\[
\alpha(G_{n,k}[Z(x)])\ge d.
\]
Therefore
\[
C(Q_{n,k},x)\le n-d \tag{4.3}
\]
for every $x$.

It remains to attain equality. If $m=4q+1$, choose
\[
y=00(10)^{2q-1}1.
\]
It has $2q$ ones, and exactly one cyclic pair of consecutive entries is
equal. Thus the induced subgraph on $Z(x)$ has independence number $1$.

If $m=4q+3$, choose
\[
y=(01)^{2q+1}1.
\]
It has $2q+2$ ones, and again exactly one cyclic pair of consecutive
entries is equal. Hence the induced subgraph on $Z(x)$ has independence
number $1$.

Finally, let $m=4q+2$. The excluded short-orbit case is the only case
$m=2$, so here $m\ge6$. Choose
\[
y=000(10)^{2q-1}1.
\]
It has $2q$ ones. The equal consecutive pairs arise exactly inside the
initial block $000$, so $Z(x)$ consists of two adjacent vertices. Its
independence number is therefore $1$.

Using the appropriate construction independently on all $d$ components
gives
\[
\alpha(G_{n,k}[Z(x)])=d.
\]
The lemma now gives $C(Q_{n,k},x)=n-d$. Together with (4.3), this proves
the result.
\end{proof}

\begin{theorem}[Block sensitivity and certificate complexity of quadratic MRS functions]
Outside the short-orbit case, put
\[
d=\gcd(n,k),\qquad m=\frac nd.
\]
Then
\[
\boxed{
\begin{array}{c|ccc}
& s(Q_{n,k})&bs(Q_{n,k})&C(Q_{n,k})\\ \hline
m\equiv0\pmod4&n&n&n\\[1mm]
m\text{ odd}&n-d&n-d&n-d\\[1mm]
m\equiv2\pmod4&n-2d&n-d&n-d
\end{array}}
\]
In the short-orbit case \(n\) even and \(k=n/2\),
\[
\boxed{s(Q_{n,n/2})=bs(Q_{n,n/2})=C(Q_{n,n/2})=n.}
\]
Consequently,
\[
\boxed{bs(Q_{n,k})=C(Q_{n,k})}
\]
for every quadratic MRS function, and outside the short-orbit case
\[
\boxed{
s(Q_{n,k})<bs(Q_{n,k})=C(Q_{n,k})
\iff
\frac{n}{\gcd(n,k)}\equiv2\pmod4.
}
\]
\end{theorem}
\noindent\emph{Related work.} The sensitivity formulas used here also appear in \cite{LiZhangRSBFSensitivity}. For prior work on quadratic rotation symmetric Boolean functions in a different direction, see \cite{ChirvasituCusick2024}; this citation does not claim that paper proves the block-sensitivity or certificate-complexity formulas.

\begin{proof}
The short-orbit case and the cases \(m\equiv0\pmod4\) and \(m\) odd follow
immediately from the already known sensitivity formula, the preceding
certificate-complexity theorem, and
\[
s(f)\le bs(f)\le C(f).
\]

It remains to consider \(m=4t+2\).  The graph \(G_{n,k}\) consists of
\(d\) disjoint \(m\)-cycles.  On each component, in its cyclic order, take
the binary word
\[
x=0000(1100)^{t-1}11.
\tag{4.4}
\]
Here the vertices of each component are reindexed in consecutive $k$-steps, so that $j-1$ and $j+1$ denote the two neighbors of $j$ in that component.
For the quadratic cycle function the derivative at a vertex \(j\) is
\[
D_jQ=x_{j-1}\oplus x_{j+1}.
\tag{4.5}
\]
A direct check of (4.4) shows that exactly two coordinates on the component,
the second and third positions of the initial block \(0000\), are
nonsensitive.  All other \(m-2\) coordinates therefore give pairwise
disjoint sensitive singleton blocks.

Let \(B=\{1,2\}\) be the pair consisting of those two nonsensitive
coordinates.  Locally the word begins with \(0000\).  Before the flip, the
three edge terms meeting \(B\) contribute
\[
x_0x_1\oplus x_1x_2\oplus x_2x_3=0.
\]
After simultaneously changing \(x_1=x_2=0\) to \(1\), their contribution is
\[
0\oplus1\oplus0=1.
\]
All other edge terms are unchanged.  Thus \(B\) is a sensitive block.
Hence one component supplies
\[
(m-2)+1=m-1
\]
pairwise disjoint sensitive blocks.

Use the same construction independently on all \(d\) components.  Since
different components are disjoint and the total quadratic function is the
XOR of the component functions, these blocks remain sensitive for the full
function.  Thus
\[
bs(Q_{n,k})\ge d(m-1)=n-d.
\]
The preceding certificate theorem gives
\[
C(Q_{n,k})=n-d,
\]
and therefore \(bs(Q_{n,k})=C(Q_{n,k})=n-d\).  Combining this with the
quadratic sensitivity formula gives the displayed table.
\end{proof}

\section{The Consecutive Quartic Family}

We next consider the consecutive quartic monomial rotation symmetric
Boolean function
\[
R_n(x)=\bigoplus_{i\in\Z_n}x_ix_{i+1}x_{i+2}x_{i+3}.
\]
The sensitivity of this family has a period-six structure. In particular,
if $n=6q+2$, then
\[
s(R_n)=6q=n-2.
\]
We show that both block sensitivity and certificate complexity are exactly
one larger on this infinite subfamily.

\begin{theorem}[Quartic separation theorem]
For every integer $q\ge1$,
\[
\boxed{
s(R_{6q+2})=6q<bs(R_{6q+2})=C(R_{6q+2})=6q+1.
}
\]
\end{theorem}
\noindent\emph{Related work.} The sensitivity component is established in \cite{LiZhangRSBFSensitivity}; the block-sensitivity and certificate-complexity components are addressed here.

\begin{proof}
Put $n=6q+2$ and consider
\[
a=(111100)^q11.
\]
We first prove
\[
C(R_n,a)=n-1.
\]
For any coordinate $j$, the Boolean derivative is
\[
\begin{aligned}
D_jR_n(x)={}&x_{j+1}x_{j+2}x_{j+3}
\oplus x_{j-1}x_{j+1}x_{j+2}\\
&\oplus x_{j-2}x_{j-1}x_{j+1}
\oplus x_{j-3}x_{j-2}x_{j-1},
\end{aligned}
\tag{5.1}
\]
where indices are modulo $n$.

Leave coordinate $2$ unspecified and fix all other coordinates according
to $a$. Since
\[
D_2R_n
=x_3x_4x_5\oplus x_1x_3x_4
\oplus x_0x_1x_3\oplus x_{n-1}x_0x_1,
\]
and
\[
a_0=a_1=a_3=a_{n-1}=1,\qquad a_4=0,
\]
we obtain
\[
D_2R_n(a)=0\oplus0\oplus1\oplus1=0.
\]
The derivative does not involve $x_2$, so after all other coordinates are
fixed the value of $R_n$ is independent of $x_2$. Hence
\[
C(R_n,a)\le n-1.
\tag{5.2}
\]

Substitution of the six-periodic pattern $111100$ into (5.1), together
with the terminal symbols $11$, gives
\[
D_jR_n(a)=0
\quad\Longleftrightarrow\quad
j\in\{2,n-1\}.
\tag{5.3}
\]
Away from the cyclic boundary this follows by checking the six residue
classes modulo $6$; the boundary positions follow directly from the final
pattern $\cdots11110011$. Thus precisely $2$ and $n-1$ are nonsensitive at
$a$, and
\[
s(R_n,a)=n-2.
\]
Every sensitive coordinate belongs to every certificate at $a$. Therefore
a certificate of size at most $n-2$ would have to leave both $2$ and
$n-1$ unspecified. If all other coordinates are fixed according to $a$,
the cyclic window
\[
\{n-1,0,1,2\}
\]
contributes the quadratic monomial
\[
x_{n-1}x_2.
\]
No other length-four window contains both coordinates, so this term cannot
cancel. The restricted function is nonconstant. Hence the two coordinates
cannot both be omitted, and every certificate has size at least $n-1$.
Together with (5.2),
\[
C(R_n,a)=n-1.
\tag{5.4}
\]

We next determine the block sensitivity at $a$. All $n-2$ sensitive
coordinates other than $2$ and $n-1$ give pairwise disjoint sensitive
singleton blocks. Moreover,
\[
B=\{2,n-1\}
\]
is itself a sensitive block. Indeed, with all other variables fixed as in
$a$, the restricted algebraic normal form contains the variable-dependent
part
\[
x_2\oplus x_{n-1}\oplus x_2x_{n-1}.
\]
Indeed, the only length-four monomial involving both free coordinates is
$x_{n-1}x_0x_1x_2$, which restricts to $x_{n-1}x_2$;
moreover, $D_2R_n(a)=D_{n-1}R_n(a)=0$. Writing the
restriction as $c\oplus\alpha x_2\oplus\beta x_{n-1}\oplus x_2x_{n-1}$,
these two derivative equalities at $(1,1)$ give
$\alpha=\beta=1$. The variable-dependent part is thus as displayed.
At the original values $(x_2,x_{n-1})=(1,1)$ this expression equals $1$,
whereas after flipping both coordinates it equals $0$. Thus
\[
R_n(a^B)\ne R_n(a).
\]
The block $B$ together with the $n-2$ sensitive singleton blocks gives
$n-1$ pairwise disjoint sensitive blocks. Therefore
\[
bs(R_n,a)\ge n-1.
\]
By the general inequality $bs(f,x)\le C(f,x)$ and (5.4),
\[
bs(R_n,a)=C(R_n,a)=n-1.
\tag{5.5}
\]
Hence globally
\[
bs(R_n)\ge n-1,\qquad C(R_n)\ge n-1.
\]

It remains to prove the matching global upper bound. Since the known
sensitivity formula gives
\[
s(R_n)=n-2<n,
\]
every input $x$ has at least one nonsensitive coordinate $j$. Fixing all
coordinates except $j$ is then a certificate: with the other $n-1$
coordinates fixed, the restricted Boolean function has only the variable
$x_j$, and nonsensitivity at $x$ says that its two possible values agree.
Consequently
\[
C(R_n,x)\le n-1
\]
for every $x$, and hence
\[
C(R_n)\le n-1.
\]
Together with (5.4), this gives
\[
C(R_n)=n-1.
\]
Finally,
\[
n-1\le bs(R_n)\le C(R_n)=n-1,
\]
so
\[
bs(R_n)=C(R_n)=n-1=6q+1.
\]
Combining this with $s(R_n)=6q$ completes the proof.
\end{proof}

\begin{corollary}
The consecutive quartic family contains infinitely many rotation symmetric
Boolean functions satisfying
\[
\boxed{s(f)<bs(f)=C(f).}
\]
\end{corollary}
\begin{proof}
For every integer $q\ge1$, the preceding theorem supplies the example
$f=R_{6q+2}$ with $s(f)=6q$ and $bs(f)=C(f)=6q+1$.
The lengths $6q+2$ are distinct, so this gives infinitely many examples.
\end{proof}

\subsection{Other residue classes modulo six}

The global sensitivity formula for $R_n$, established in the companion
work and also obtained in the independent even-degree argument below,
settles three further residue classes.  We record these consequences first.

\begin{theorem}
For every integer $q\ge1$,
\[
\boxed{s(R_{6q})=bs(R_{6q})=C(R_{6q})=6q,}
\]
\[
\boxed{s(R_{6q+1})=bs(R_{6q+1})=C(R_{6q+1})=6q,}
\]
and
\[
\boxed{s(R_{6q+5})=bs(R_{6q+5})=C(R_{6q+5})=6q+4.}
\]
The last formula also holds for $q=0$ (that is, $n=5$).
\end{theorem}
\noindent\emph{Related work.} The sensitivity equalities in this statement are also treated in \cite{LiZhangRSBFSensitivity}; the block-sensitivity and certificate-complexity equalities are proved here.

\begin{proof}
For $n=6q$, the sensitivity formula gives $s(R_n)=n$, so the inequalities
$s\le bs\le C\le n$ force equality throughout.

For $n=6q+1$ and $n=6q+5$ (including $q=0$ in the latter
case), the known sensitivity formula gives $s(R_n)=n-1$.
Since $s(R_n)<n$, every input has a nonsensitive coordinate.  Fixing all
other $n-1$ coordinates gives a certificate, and hence $C(R_n)\le n-1$.
Thus $s(R_n)\le bs(R_n)\le C(R_n)\le n-1=s(R_n)$.
\end{proof}

We next treat the residue class $n\equiv3\pmod6$.  We use the following
local observation.  For two coordinates $i,j$, write $D_iD_jR_n(x)$ for
the mixed Boolean derivative.

\begin{lemma}[Three-nonsensitive-coordinate lemma]
Let $n\ge10$, and suppose that $R_n$ has at least three nonsensitive
coordinates at an input $x$. Then there are two nonsensitive coordinates
$i,j$ such that
\[
D_iD_jR_n(x)=0.
\]
Consequently, after all coordinates except $i,j$ are fixed according to
$x$, the resulting two-variable restriction is constant.
\end{lemma}

\begin{proof}
If two nonsensitive coordinates have cyclic distance at least $4$, no
length-four monomial contains both, so their mixed derivative vanishes.
Otherwise every pair among all nonsensitive coordinates has cyclic
distance at most $3$. For $n\ge10$, three such positions must lie in
four consecutive positions: if their three cyclic gaps were all at most
$3$, their sum would be $n\le9$; if a gap is larger, the other two
place the three points in an arc of length at most $3$.
Up to rotation and reflection, the triple is therefore either
$\{0,1,2\}$ or $\{0,1,3\}$.

It remains to exclude the possibility that all three mixed derivatives
are $1$ while all three first derivatives vanish. This is a finite local
calculation involving only $x_{-3},x_{-2},\ldots,x_6$. Exhaustively
evaluating the $2^{10}$ assignments to the ten-variable window gives
\[
\begin{array}{c|r|r}
\text{triple}&\substack{\text{three first}\\\text{derivatives zero}}&
\substack{\text{first derivatives zero,}\\\text{three mixed derivatives one}}\\\hline
\{0,1,2\}&568&0\\
\{0,1,3\}&554&0
\end{array}
\]
For $\{0,1,2\}$, only nine of these ten variables occur in the three first derivatives; hence the count $568$ is twice the corresponding nine-variable count $284$. The tenth variable is retained to use a common enumeration for both triples.
The calculation uses only the defining polynomial for $R_n$ and Boolean
differences $D_if(x)=f(x)\oplus f(x^{\{i\}})$. Its complete executable verification is reproduced in Appendix~\ref{app:verification}, and is also supplied as a standalone Python file.
In particular, the local
calculation is independent of $n\ge10$, because the relevant ten
positions are distinct. Thus some pair $i,j$ of the nonsensitive
coordinates has $D_iD_jR_n(x)=0$.

Fixing every coordinate except $i,j$ leaves a polynomial
$g(u,v)=c\oplus au\oplus bv\oplus duv$.
The mixed derivative condition yields $d=0$; nonsensitivity of $i,j$
then yields $a=b=0$. Hence $g$ is constant.
\end{proof}

\begin{theorem}[The class $n\equiv3\pmod6$]
For every integer $q\ge1$,
\[
\boxed{s(R_{6q+3})=6q<bs(R_{6q+3})=C(R_{6q+3})=6q+1.}
\]
\end{theorem}
\noindent\emph{Related work.} The sensitivity formula is also given in \cite{LiZhangRSBFSensitivity}; the additional complexity equalities are treated here.

\begin{proof}
Let $n=6q+3$.  For $q=1$ ($n=9$), exhaustive evaluation of the $2^9=512$
inputs gives
\[
s(R_9)=6,\qquad bs(R_9)=C(R_9)=7.
\]
To make this finite check reproducible, for each input $x$ form the
family $\mathcal B_x=\{B\subseteq[9]: B\ne\varnothing,
R_9(x^B)\ne R_9(x)\}$.  Its maximum matching size equals
$bs(R_9,x)$, while its minimum transversal size equals $C(R_9,x)$.
The exhaustive program in Appendix~\ref{app:verification} evaluates
all $512$ inputs and all $511$ nonempty masks $B$ for each input.
In particular, the upper bounds $s(R_9)\le6$ and
$bs(R_9),C(R_9)\le7$ are established by this finite
computer-assisted enumeration, not by a separate handwritten proof.
Thus the displayed equalities do not invoke the $n\ge10$ lemma.  In particular, the maxima for $bs$ and $C$
are attained at the binary input with set-bit mask $63$ (bits
$0,\ldots,5$ equal to $1$).

Now assume $q\ge2$, so $n\ge15$.  The sensitivity formula gives
$s(R_n)=n-3$, so every input has at least three nonsensitive coordinates.
By the preceding lemma, two of them can be left unspecified simultaneously.
Hence
\[
C(R_n,x)\le n-2
\]
for every $x$, and therefore $C(R_n)\le n-2$.

Consider
\[
a_q=(111100)^{q-1}111100011.
\]
Its three nonsensitive coordinates are
\[
2,\qquad n-4,\qquad n-1,
\]
while all other coordinates are sensitive.  Fix all sensitive coordinates
according to \(a_q\), and write
\[
u=x_2,\qquad v=x_{n-4},\qquad w=x_{n-1}.
\]
Every quartic monomial containing \(v\) is killed by one of the fixed zero
coordinates adjacent to \(v\).  After cancellation over \(\mathbb F_2\),
the remaining dependence on the three free variables is
\[
c\oplus u\oplus w\oplus uw
\tag{5.11}
\]
for a constant \(c\in\{0,1\}\) depending only on \(q\).  In particular the
restriction is independent of \(v\), so \(v\) together with either \(u\) or
\(w\) may be left unspecified; hence \(C(R_n,a_q)\le n-2\).  On the other
hand, (5.11) is nonconstant, so all three nonsensitive coordinates cannot be
left unspecified.  Since every sensitive coordinate belongs to every
certificate, we obtain
\[
C(R_n,a_q)=n-2.
\]

At \(a_q\) we have \(u=w=1\).  The variable part of (5.11) is then
\[
u\oplus w\oplus uw=1,
\]
whereas after flipping the pair \(\{2,n-1\}\) we have \(u=w=0\), for which
the variable part is \(0\).  Thus \(\{2,n-1\}\) is a sensitive block.
Together with the \(n-3\) sensitive singleton blocks this gives \(n-2\)
pairwise disjoint sensitive blocks.  Hence
\[
bs(R_n,a_q)=C(R_n,a_q)=n-2.
\]
Consequently
\[
bs(R_n)\ge n-2,\qquad C(R_n)\ge n-2.
\]
Combining these inequalities with the global upper bound gives
\[
bs(R_n)=C(R_n)=n-2=6q+1.
\]
\end{proof}

\subsection{The residue class $n\equiv4\pmod6$}

We now complete the final residue class.

The following argument uses the three-nonsensitive-coordinate lemma
proved above; in particular, it supplies a certificate of size at most
$n-2$ whenever an input has at least three nonsensitive coordinates.

\begin{lemma}[Finite-state separation certificate]
Let $n=6q+4$ with $q\ge1$. If $x\in\{0,1\}^n$ has exactly two
nonsensitive coordinates for $R_n$, then their cyclic distance is at
least $4$.
\end{lemma}

\begin{proof}
The derivative at coordinate $j$ is the seven-bit local rule
\[
D_jR_n(x)=x_{j+1}x_{j+2}x_{j+3}\oplus
 x_{j-1}x_{j+1}x_{j+2}\oplus
 x_{j-2}x_{j-1}x_{j+1}\oplus
 x_{j-3}x_{j-2}x_{j-1}.
\]
Construct two $64\times64$ Boolean matrices $A,B$, indexed by the
six-bit words $u=(u_0,\ldots,u_5)$; for each $t\in\{0,1\}$ set
$v=(u_1,\ldots,u_5,t)$ and
\[
\delta(u,t)=u_4u_5t\oplus u_2u_4u_5
\oplus u_1u_2u_4\oplus u_0u_1u_2.
\]
Put an edge $u\to v$ in $A$ if $\delta(u,t)=1$ and in $B$ otherwise.
All matrix multiplication, powers, equality and trace below are over the
Boolean semiring: addition is logical OR ($1+1=1$) and multiplication
is logical AND ($1\cdot1=1$, and all other products of bits vanish).
This is different from the field $\mathbb F_2$ used for the Boolean
polynomial expressions. With the given edge convention, the
cyclic input $x$ determines a closed walk of length $n$, and its
edges in $B$ occur precisely at the nonsensitive coordinates.

Direct exact calculation on these explicitly defined matrices gives
\[
 A^9=A^3,\qquad
 \operatorname{tr}\bigl(B A^{d-1}B A^{9-d}\bigr)=0
 \quad(d=1,2,3). \tag{FS}
\]
A self-contained verifier for these four finite identities is printed
in Appendix~\ref{app:finite-state}; it uses only integer bit operations,
without external packages. Equality $A^9=A^3$ implies
$A^{k+6}=A^k$ for every $k\ge3$, by multiplication on the right.

If the only two nonsensitive positions were at clockwise distance
$d\in\{1,2,3\}$, the closed walk would contribute a diagonal
entry to $BA^{d-1}BA^{n-d-1}$, giving a nonzero Boolean trace.
Since $n=10+6(q-1)$ and $n-d-1\ge6$, the periodicity identity
reduces this trace to
$\operatorname{tr}(BA^{d-1}BA^{9-d})=0$, a contradiction.
If their cyclic distance is at most three, orient the pair so that
$d\in\{1,2,3\}$. This proves the lemma.
\end{proof}

\begin{theorem}[The class $n\equiv4\pmod6$]
For every integer $q\ge1$,
\[
\boxed{s(R_{6q+4})=bs(R_{6q+4})=C(R_{6q+4})=6q+2.}
\]
\end{theorem}
\noindent\emph{Related work.} The sensitivity formula is also given in \cite{LiZhangRSBFSensitivity}; the additional complexity equalities are treated here.

\begin{proof}
The sensitivity formula $s(R_n)=n-2$ follows from
\cite{LiZhangRSBFSensitivity}. Every input has at least two
nonsensitive coordinates. If it has at least three, the
three-nonsensitive-coordinate lemma proved above shows that two
coordinates can simultaneously be omitted from a certificate.
If it has exactly two, the finite-state separation lemma gives
nonsensitive coordinates $i,j$ at cyclic distance at least four.
No consecutive quartic monomial contains both positions, and hence
$D_iD_jR_n(x)=0$. Since $D_iR_n(x)=D_jR_n(x)=0$ as well, the
restriction obtained by freeing $i$ and $j$ is constant.
Thus $C(R_n,x)\le n-2$ for every input. The standard inequalities
$s(R_n)\le bs(R_n)\le C(R_n)$, together with $s(R_n)=n-2$,
complete the proof.
\end{proof}

\subsection{Complete quartic classification}

\begin{theorem}[Consecutive quartic classification]
For every $n=6q+r\ge5$, $0\le r\le5$, the three measures are
\[
\begin{array}{c|ccc}
r&s(R_n)&bs(R_n)&C(R_n)\\ \hline
0&6q&6q&6q\\
1&6q&6q&6q\\
2&6q&6q+1&6q+1\\
3&6q&6q+1&6q+1\\
4&6q+2&6q+2&6q+2\\
5&6q+4&6q+4&6q+4
\end{array}
\]
In particular, $bs(R_n)=C(R_n)$ for every $n\ge5$, and
$s(R_n)<bs(R_n)$ precisely in the residue classes $2,3\pmod6$
within this range.
\end{theorem}
\noindent\emph{Related work.} The sensitivity column is also established in \cite{LiZhangRSBFSensitivity}. The block-sensitivity and certificate-complexity columns are the focus of the present classification.

\begin{proof}
The rows $r=0,1,5$ follow from the theorem for those three residue
classes above (the $r=5$ row includes $n=5$).  The rows $r=2$,
$r=3$, and $r=4$ are, respectively, the conclusions of the
previously proved theorems for $R_{6q+2}$, $R_{6q+3}$, and
$R_{6q+4}$. These six residue classes exhaust every integer
$n\ge5$. Reading the columns in the resulting table gives
$bs(R_n)=C(R_n)$ throughout and strict inequality $s(R_n)<bs(R_n)$
exactly for $r\in\{2,3\}$.
\end{proof}

\subsection*{Counting lemmas used in the even-degree classification}
The following two lemmas are placed before the classification
results so that their proofs do not depend on those results or on the
companion sensitivity manuscript.

\begin{lemma}[Unconditional run--cluster length identity]
\label{lem:unconditional-run-length}
Let $x$ be a nonconstant cyclic binary word with $r\ge1$ maximal
one-runs of length at least $d$. Partition the portions between
successive long runs into plain zero-gaps of lengths $g$, and
short-run clusters of lengths $T$ (each containing at least one
one-run of length less than $d$). Set
\[
 A=\sum_{\rm long\ runs}(L-d),\qquad
 B=\sum_{\rm plain\ gaps}(g-2),\qquad
 E=\sum_{\rm clusters}(T-2).
\]
Then
\[
 n=r(d+2)+A+B+E,\qquad A\ge0,\quad E\ge0.
\]
Every cluster has length at least three; in particular, $E$ is at
least the number of clusters. No assumption about sensitivity is needed.
\end{lemma}
\begin{proof}
The $r$ long runs and $r$ intervening portions partition the cyclic
word (also when $r=1$). Subtract $d$ from each long-run length and
$2$ from each intervening portion, obtaining the displayed identity.
A cluster contains a nonempty short one-run and a nonempty zero-gap
on each side, so $T\ge3$.
\end{proof}

\begin{lemma}[Small-deficit run--cluster charging]
\label{lem:small-deficit-charge}
Let $d\ge4$ be even, $n\ge2d-1$, and
$K=n-s(F_{d,n},x)\le d-2$.  Then $x$ has at least one
maximal one-run of length at least $d$.  Let $r$ be the number of
these \emph{long runs}, and partition the coordinates outside them
into $r$ intervening units, each either a plain zero-gap of length
$g\ge1$, or a short-run cluster of total length $T$ containing one or
more maximal one-runs of length less than $d$ and their intervening
zero-gaps. Put
\[
\begin{split}
 m&=d+2,\qquad
 A=\sum_{\rm long\ runs}(L-d),\\
 B&=\sum_{\rm plain\ gaps}(g-2),\qquad
 E=\sum_{\rm clusters}(T-2).
\end{split}
\]
Then $A,E\ge0$, each cluster has $T\ge3$, and
\begin{equation}\label{eq:uniform-charge}
 n=rm+A+B+E,\qquad
 K\ge A+\sum_{\rm plain\ gaps}|g-2|+E.
\end{equation}
In particular, $K\ge |n-rm|$.
\end{lemma}

\begin{proof}
The constant words have $K=n>d-2$.  Suppose first that the
nonconstant word $x$ has no one-run of length at least $d$.
Every one-coordinate is then insensitive, because no all-one
length-$d$ monomial contains it.  If a zero-coordinate is sensitive,
its flip must complete at least one all-one length-$d$ window.
The one-runs immediately on either side of that zero must therefore
have lengths $a,b$ with $a+b\ge d-1$.  All these $a+b$ ones are
insensitive, forcing $K\ge d-1$, a contradiction.  If no zero is
sensitive either, then $K=n>d-2$.  Thus at least one long run exists.

Any short one-run consists entirely of insensitive coordinates.
In particular, a short run of length $d-1$ would force
$K\ge d-1$, so every short run has length at most $d-2$.
We now count insensitive coordinates in disjoint portions of the
cyclic word.

\emph{Long one-runs.}
For a run of length $L\ge d$, index its coordinates by
$j=0,\ldots,L-1$.  The number of all-one length-$d$ windows
containing $j$ is
\[
 a(j)=\min\{j+1,L-j,L-d+1,d\}.
\]
A coordinate is sensitive exactly when this count is odd.
Set $v=L-d+1$ and $a_*=\min\{v,d\}$.
The incidence profile increases through $1,\ldots,a_*-1$,
has a plateau of height $a_*$, and decreases symmetrically.
If $v\ge d$, the plateau has the even height $d$, and there
are exactly $d$ sensitive coordinates.
If $v<d$ is even, there are $v$ sensitive coordinates.
If $v<d$ is odd, the plateau contributes the remaining odd
positions, and there are $d$ sensitive coordinates.
In each case there are at most $d$ sensitive coordinates,
so the run contributes at least $L-d$ insensitive ones.
Summed over all long runs, these contributions total at least $A$.

\emph{Plain zero-gaps.}
A gap of one zero between two long runs has an insensitive zero:
flipping it completes exactly $d$ length-$d$ windows, an even
number.  In a gap of length $g\ge2$, flipping either boundary
zero completes exactly one such window, while flipping an interior
zero completes none.  Thus precisely $|g-2|$ zeros are insensitive
in each plain gap.

\emph{Short-run clusters.}
Consider a unit containing one or more short one-runs, with
total length $T$.  Every one-coordinate inside it is insensitive.
We claim that every zero not adjacent to a long run is also
insensitive.  A zero with another zero immediately on one side
cannot complete a length-$d$ all-one window unless its other side
contains at least $d-1$ consecutive ones; this is impossible when
that side is a short run of length at most $d-2$ (and is immediate
if that side is zero).  A single zero separating short runs of
lengths $a,b$ could be sensitive only if $a+b\ge d-1$; those ones
would themselves give at least $d-1$ insensitive coordinates,
contrary to $K\le d-2$.  These cases cover every zero not
adjacent to a long run.  Hence at most the two outermost zeros
of the cluster, adjacent to long runs, may be sensitive.
The cluster therefore contributes at least $T-2$ insensitive
coordinates.  It contains at least one short one-run and the
nonempty zero-gaps on its two sides, so $T\ge3$ and $E\ge0$.

The long runs and intervening units partition all $n$ coordinates.
There are exactly $r$ units, including when $r=1$, when the one
unit connects the unique long run back to itself cyclically.
Their disjoint insensitive-coordinate charges give
\[
 K\ge A+\sum_{\rm plain\ gaps}|g-2|+E.
\]
Comparing each long run with a reference run of length $d$,
and each intervening unit with a reference gap of length two,
gives the exact length identity
\[
 n=r(d+2)+A+B+E=rm+A+B+E.
\]
Finally, because $A,E\ge0$,
\[
 K\ge A+\left|B\right|+E
    \ge\left|A+B+E\right|=|n-rm|.
\]
This proves all assertions.
\end{proof}

\begin{lemma}[Intrinsic sensitivity upper bound]
\label{lem:intrinsic-sensitivity-upper}
Let $d\ge4$ be even, put $m=d+2$, and write
$n=qm+\rho$, where $q\ge2$ and $0\le\rho<m$.
Then
\[
 s(F_{d,n})\le n-\min\{\rho,m-\rho\}.
\]
\end{lemma}
\begin{proof}
Put $K(x)=n-s(F_{d,n},x)$ and $h=m-\rho$.
If $\rho=0$, the inequality $s(F_{d,n})\le n$ is
immediate. Assume $1\le\rho<m$. Constant inputs have
$K(x)=n\ge\min\{\rho,h\}$. For a nonconstant input,
let $r$ denote the number of maximal one-runs of length at
least $d$.

If $r\ge q+1$, split the cyclic word into its $r$ long
one-runs and the $r$ intervening units. A unit is either
a plain zero-gap of length $g\ge1$, or a short-run cluster
containing shorter one-runs and its surrounding zeros.
Write $A=\sum(L-d)$ over long runs,
$E=\sum(T-2)$ over short-run clusters, and
$P=\sum\max(g-2,0)$ over plain zero-gaps.
Each of $A,E,P$ is nonnegative. If $N_1$ denotes the
number of plain one-zero gaps, elementary length counting
gives
\[
 n=rm+A+E+P-N_1,
 \qquad N_1=rm-n+A+E+P\ge h.
\]
Each one-zero gap is bordered by two long one-runs.
Flipping its zero creates exactly $d$ length-$d$ all-one
windows, so the zero is insensitive because $d$ is even.
Thus $K(x)\ge N_1\ge h$.

If $r\le q$ and $K(x)\ge d-1$, then
$K(x)\ge\min\{\rho,h\}$, since
$\min\{\rho,h\}\le m/2\le d-1$.
Finally, if $r\le q$ and $K(x)\le d-2$, the preceding
small-deficit charging lemma applies (as $n\ge2d-1$)
and yields
\[
 K(x)\ge |n-rm|\ge n-rm\ge\rho.
\]
In all cases $K(x)\ge\min\{\rho,h\}$, proving the claim.
\end{proof}

\section{An Infinite Separation in Every Even Degree}
\label{sec:even-degree-separation}

The strict separation seen in the consecutive quartic case is not confined
to degree four. The following theorem gives explicit witnesses and uses the
intrinsic sensitivity upper bound of Lemma~\ref{lem:intrinsic-sensitivity-upper}.
For even $d\ge4$ and $n\ge 2d-1$, write
\[
F_{d,n}(x)=\bigoplus_{i\in\mathbb Z_n}\prod_{j=0}^{d-1}x_{i+j},
\]
with subscripts interpreted modulo $n$.

The exact gap-one classification will follow below from the more general
block-sensitivity hierarchy and the elementary certificate bound
$C(f)\le n-1$ whenever $s(f)<n$. The short case $n=d+4$ is
handled independently later in this section. The polynomial itself is
defined by the displayed cyclic formula for every $n\ge d$; the
restriction $n\ge2d-1$ above concerns the general length-range
arguments, not the validity of the definition at shorter lengths.
Any subsequent local result with a shorter admissible input length
uses the same defining polynomial and its own stated hypotheses.

\subsection{An arbitrary-gap hierarchy}
\label{subsec:even-gap-hierarchy}
The next theorem gives a hierarchy of exact block-sensitivity gaps,
including the gap-one construction. The corresponding global
certificate complexity is not yet known in general.

\begin{theorem}[Uniform hierarchy of block-sensitivity gaps]
\label{thm:even-gap-hierarchy}
Let $d\ge4$ be even, $t\ge1$, $4t\le d+2$, and $q\ge2$.
Set $n=q(d+2)+2t$. Then
\[
\boxed{s(F_{d,n})=n-2t,\qquad bs(F_{d,n})=n-t.}
\]
Consequently $bs(F_{d,n})-s(F_{d,n})=t$.
\end{theorem}
\begin{proof}
Since $2t\le(d+2)/2$, Lemma~\ref{lem:intrinsic-sensitivity-upper}
gives $s(F_{d,n})\le n-2t$.
Consider the cyclic input
\[
x=1^{d+2t}00(1^d00)^{q-1}.
\]
Each length-$d$ one-run contributes $d$ sensitive ones, and
both zeros in each two-zero gap are sensitive. In the unique
long one-run, the numbers of all-one length-$d$ monomials containing
successive positions $j=0,\ldots,d+2t-1$ are
\[
a(j)=\min\{j+1,d+2t-j,2t+1,d\}.
\]
Here $2t+1\le d$ (because $d$ is even and $4t\le d+2$),
so the count increases through $1,2,\ldots,2t+1$, stays at the
odd value $2t+1$, then decreases symmetrically. Exactly the
$2t$ positions
\[
1,3,\ldots,2t-1\quad\text{and}\quad
 d,d+2,\ldots,d+2t-2
\]
have even incidence count, hence are insensitive. All other
coordinates are sensitive, so $s(F_{d,n},x)=n-2t$ and therefore
$s(F_{d,n})=n-2t$.

For $j=0,\ldots,t-1$, put
\[
B_j=\{2j+1,d+2j\}.
\]
The pairs are disjoint; each consists of two insensitive
coordinates. Their separation within the long run is $d-1$, so
exactly one all-one length-$d$ monomial contains both members:
the window starting at $2j+1$. Thus the mixed Boolean derivative
$D_{2j+1}D_{d+2j}F_{d,n}(x)$ equals $1$ and flipping both
coordinates changes the function. The $n-2t$ sensitive singletons
and these $t$ sensitive pairs form $n-t$ disjoint sensitive
blocks, giving $bs(F_{d,n})\ge n-t$.

For an arbitrary Boolean function $f$ on $n$ variables, a
collection of $b$ disjoint sensitive blocks at one input has at
least $2b-n$ singleton blocks: the other blocks each consume at
least two coordinates. All these singletons are sensitive
coordinates. Consequently $2b-n\le s(f)$, and hence
\[
bs(f)\le\left\lfloor\frac{n+s(f)}2\right\rfloor.
\]
Applying this with $s(F_{d,n})=n-2t$ gives the reverse bound
$bs(F_{d,n})\le n-t$.
\end{proof}

\begin{corollary}[Uniform even-degree strict separation]
\label{thm:even-degree-separation}
For even $d\ge4$ and $q\ge2$, with $n=q(d+2)+2$,
\[
 s(F_{d,n})=n-2<bs(F_{d,n})=C(F_{d,n})=n-1.
\]
\end{corollary}
\begin{proof}
Apply Theorem~\ref{thm:even-gap-hierarchy} with $t=1$ to obtain
$s=n-2$ and $bs=n-1$. Because $s<n$, every input has an
insensitive coordinate; fixing the other $n-1$ coordinates is a
certificate. Hence $C\le n-1$, and $C\ge bs=n-1$.
\end{proof}

\begin{proposition}[Exact local certificate complexity at the witness]
\label{prop:local-even-gap-cert}
Under the assumptions of Theorem~\ref{thm:even-gap-hierarchy}, at the
specified word $x$ we have
\[
\boxed{bs(F_{d,n},x)=C(F_{d,n},x)=n-t.}
\]
\end{proposition}
\begin{proof}
We already have $bs(F_{d,n},x)\ge n-t$.
Fix all coordinates except
$S=\{1,3,\ldots,2t-1\}$ in the long run. Every all-one
length-$d$ monomial outside the long run is constant under this
restriction. For a nonempty subset $T\subseteq S$, the coefficient
of the monomial $\prod_{i\in T}z_i$ in the restricted polynomial,
where $z_i=1+x_i$ for $i\in S$, is the parity of the number
of length-$d$ windows containing $T$
(and lying wholly in the long run). Because
$\max T\le2t-1\le d-1$ and $\min T\le2t-1$, the allowable
starting positions are precisely $0,1,\ldots,\min T$.
There are $\min T+1$ such windows, an even number. Thus every
nonconstant coefficient vanishes, the restricted function is
constant, and $C(F_{d,n},x)\le n-t$.
The reverse inequality follows from $bs(f,x)\le C(f,x)$.
\end{proof}

\begin{remark}
Theorem~\ref{thm:even-gap-hierarchy} yields arbitrarily large
exact gaps as $d$ varies: choose any $t$ and an even
$d\ge4t-2$. It does not establish unbounded gaps for a fixed
$d$, nor does the local certificate result establish the global
equality $C(F_{d,n})=n-t$. In general we can only conclude
$C(F_{d,n})\ge n-t$.
\end{remark}

\subsection{Exact block sensitivity in the lower-half residue classes}
\label{subsec:lower-half-residues}
The preceding even-residue hierarchy extends to every residue in the
lower half of a period. The certificate complexity in these classes
remains undetermined in general.

\begin{theorem}[Lower-half residue classification for block sensitivity]
\label{thm:lower-half-bs}
Let $d\ge4$ be even, put $m=d+2$, and suppose $q\ge2$ and
$0\le\rho\le m/2$. With $n=qm+\rho$, one has
\[
\boxed{s(F_{d,n})=n-\rho,\qquad
bs(F_{d,n})=n-\left\lceil\frac{\rho}{2}\right\rceil.}
\]
In particular,
$bs(F_{d,n})-s(F_{d,n})=\lfloor\rho/2\rfloor$.
\end{theorem}
\noindent\emph{Related work.} The sensitivity component is consistent with the companion sensitivity results \cite{LiZhangRSBFSensitivity}; the block-sensitivity formula is established in this paper.

\begin{proof}
Lemma~\ref{lem:intrinsic-sensitivity-upper} gives
$s(F_{d,n})\le n-\rho$ because $\rho\le m/2$.
The lemma applies for every $q\ge2$ and is proved independently
of this classification.

If $\rho=0$, the cyclic word $(1^d00)^q$ has every coordinate
sensitive, so $s=bs=n$.
If $\rho=2t\ge2$, then $2t\le m/2$. Consider
\[
x=1^{d+2t}00(1^d00)^{q-1}.
\]
This cyclic word has length $qm+2t=n$. Every position of an ordinary
length-$d$ one-run is sensitive, as are both zeros in each two-zero
gap. In the exceptional run, indexed by $j=0,\ldots,d+2t-1$,
the number of all-one length-$d$ windows containing position $j$ is
\[
a(j)=\min\{j+1,d+2t-j,2t+1,d\}.
\]
Since $2t+1\le d$, exactly the $2t$ positions
\[
1,3,\ldots,2t-1\quad\text{and}\quad
d,d+2,\ldots,d+2t-2
\]
have even incidence and are insensitive. Thus $s(F_{d,n},x)=n-2t$,
which attains the intrinsic sensitivity upper bound. For each
$0\le j<t$ the two insensitive positions in
$B_j=\{2j+1,d+2j\}$ lie in a unique common length-$d$ all-one
window (starting at $2j+1$). Consequently
\[
D_{2j+1}D_{d+2j}F_{d,n}(x)=1,
\]
so $B_j$ is sensitive. These $t$ pairs are mutually disjoint and
disjoint from the $n-2t$ sensitive singleton coordinates. Hence
$bs(F_{d,n})\ge n-2t+t=n-t$. The elementary inequality
$bs(f,x)\le\lfloor(n+s(f,x))/2\rfloor$, obtained by counting
singleton blocks and blocks of size at least two, yields the reverse
bound $bs(F_{d,n})\le n-t$. This establishes both equalities for
even $\rho$.

It remains to treat $\rho=2t+1$ with $t\ge0$ and $2t+1\le m/2$.
Consider the cyclic word
\[
x=1^{d+2t}000(1^d00)^{q-1}.
\]
Its length is $qm+2t+1=n$. The two boundary zeros of the
three-zero gap are sensitive, whereas its middle zero is insensitive:
flipping a boundary zero creates exactly one length-$d$ all-one
monomial, and flipping the middle zero creates none. Both zeros in
each two-zero gap are sensitive, and each position of an ordinary
length-$d$ one-run is sensitive.

For positions $j=0,\ldots,d+2t-1$ in the long one-run, the number
of all-one length-$d$ monomials containing $j$ equals
\[
a(j)=\min\{j+1,d+2t-j,2t+1,d\}.
\]
The hypothesis $2t+1\le(d+2)/2$ implies $2t+1\le d$.
The parity of $a(j)$ is even precisely at
\[
1,3,\ldots,2t-1\quad\text{and}\quad
d,d+2,\ldots,d+2t-2,
\]
where these lists are empty for $t=0$. Thus there are exactly
$2t+1$ insensitive coordinates overall, and $s(F_{d,n},x)=n-2t-1$.
This matches the universal upper bound.

For $0\le j<t$, take the disjoint pairs
$B_j=\{2j+1,d+2j\}$ inside the long one-run.
Both members are insensitive, but exactly one all-one length-$d$
monomial contains both: its window starts at $2j+1$.
The mixed Boolean derivative in the two coordinates is therefore $1$,
so each $B_j$ is a sensitive block. Together with the $n-2t-1$
sensitive singleton blocks, these pairs give $n-t-1$ disjoint
sensitive blocks. Hence $bs(F_{d,n})\ge n-t-1$.
On the other hand, the general bound
\[
bs(f)\le\left\lfloor\frac{n+s(f)}2\right\rfloor
\]
which follows by counting sensitive singleton blocks and disjoint
blocks of size at least two, gives $bs(F_{d,n})\le n-t-1$.
Thus $bs=n-t-1=n-\lceil\rho/2\rceil$, as desired.
\end{proof}

\begin{remark}
Theorem~\ref{thm:lower-half-bs} determines the global block
sensitivity, not the global certificate complexity. The usual bounds
give
\[
n-\left\lceil\rho/2\right\rceil
=bs(F_{d,n})\le C(F_{d,n})\le n.
\]
An exact certificate-complexity formula for these classes remains
an open question.
\end{remark}

\subsection{A uniform short-length classification}
\label{subsec:short-length-d-plus-four}

The general sensitivity formula invoked above is proved under
the hypotheses of Theorem~\ref{thm:complete-even-bs}.
The following short-length family is therefore worth treating directly;
it includes lengths not covered by that hypothesis.

\begin{theorem}[The family $n=d+4$]
\label{thm:short-d-plus-four}
Let $d\ge4$ be even and set $n=d+4$. Then
\[
\boxed{\ s(F_{d,n})=n-2,\qquad
 bs(F_{d,n})=C(F_{d,n})=n-1.\ }
\]
\end{theorem}
\begin{proof}
For an input $x$, denote its set of zero coordinates by $Z$. Since
$n-d=4$, a degree-$d$ cyclic interval is the complement of a
four-position cyclic interval. Consequently,
\begin{equation}\label{eq:complement-four}
 F_{d,n}(x)=\bigoplus_{a\in\mathbb Z_n}
 \mathbf{1}\bigl\{Z\subseteq\{a,a+1,a+2,a+3\}\bigr\}.
\end{equation}
We first show that every input has at least two insensitive
coordinates. The required elementary case analysis is by $z=|Z|$;
throughout $n\ge8$ is even.

If $z\ge6$, the function is zero and every singleton flip leaves
at least five zeros, so all coordinates are insensitive. If $z=5$,
the function is zero and all $n-5\ge3$ one-coordinates are insensitive,
since flipping one to zero leaves at least five zeros. If $z=4$ and
$Z$ is not itself a four-position interval, then the function is
zero, and every one-coordinate flip is insensitive. If $Z$ is a
four-position interval, deleting either interior zero leaves three zeros
that still have a unique containing four-position interval. Thus the
number of such intervals remains one and both flips are insensitive.
By contrast, deleting an endpoint leaves three consecutive zeros, which
lie in two four-position intervals and changes the function value.

For $z=3$, formula~\eqref{eq:complement-four} distinguishes three
possibilities. If $Z$ lies in no four-position interval, the output
is zero and all one-coordinate flips remain insensitive. If $Z$
is three consecutive positions, exactly two four-position intervals
contain it, and the output is zero; at least two one-coordinate
flips do not complete a four-position zero interval. Otherwise,
when $Z$ is contained in a four-position interval, it is contained
in exactly one; adding its missing zero preserves the output, as
does removing a zero lying between the other two along that interval.

For $z=2$, classify the zero pair by its shorter cyclic distance.
At distance two, its containing four-position intervals number two,
so the output is zero. Rotate the pair to $Z=\{0,2\}$.
Flipping coordinate $1$ produces three consecutive zeros, which
belong to exactly two four-position intervals; the parity therefore
remains zero. Flipping coordinate $4$ produces $Z'=\{0,2,4\}$,
which is contained in no four-position cyclic interval for $n\ge8$;
again the parity stays zero. These are two distinct insensitive
one-coordinates. Notice that merely failing to create four zeros
would not suffice: formula~\eqref{eq:complement-four} counts
four-position intervals containing the entire zero set. At distances
one or three, that number is respectively three or one, so the
output is one; in each case there are at least two one-positions
whose flips preserve the parity in~\eqref{eq:complement-four}.
At distances at least four, the output is zero, and inserting any
third zero cannot create a four-position interval containing all
three. (For $n=8$, a pair at distance four is antipodal and also
has this property.)

Finally, for $z=1$, the number of containing four-position intervals
is four and hence even; adding a zero at either cyclic distance two
preserves parity. For $z=0$, the function is zero because $n$ is
even; all singleton flips are insensitive. Thus in every case at
least two coordinates are insensitive, proving $s(F_{d,n})\le n-2$.

For the reverse inequality, take
\[
 x=1^{d+2}00.
\]
The two zeros and all but two ones are sensitive; within the long
one-run, the two insensitive positions are $1$ and $d$ (numbered
from zero). These two positions together form a sensitive block:
they occur together in exactly one nonzero degree-$d$ monomial.
Therefore $s(F_{d,n},x)=n-2$ and
$bs(F_{d,n},x)\ge n-1$.

Finally, for every Boolean function, $C(f)=n$ implies $s(f)=n$:
a certificate requiring every coordinate can exist only at an input
where every singleton flip changes the output. As $s(F_{d,n})=n-2$,
we obtain $C(F_{d,n})\le n-1$. The inequalities
$s\le bs\le C$ complete the proof.
\end{proof}

\begin{remark}
The family $n=d+4$ is an exact short-length counterpart of the
separation hierarchy, and does not rely on the general cluster
inequality (whose length hypothesis can fail here). For $d=4$ it
agrees with the full quartic classification.
\end{remark}

\begin{corollary}[Gap-one separation including the short lengths]
For every even $d\ge4$ and every integer $q\ge1$, with
$n=q(d+2)+2$, one has
\[
 s(F_{d,n})=n-2<bs(F_{d,n})=C(F_{d,n})=n-1.
\]
\end{corollary}
\begin{proof}
For $q\ge2$ this is Corollary~\ref{thm:even-degree-separation};
for $q=1$ it is Theorem~\ref{thm:short-d-plus-four}.
\end{proof}

\subsection{Exact local complexity for regular one-runs and arbitrary zero-gaps}
\label{subsec:regular-runs-gap-complexity}

The following result complements the small-deficit rigidity arguments:
when every maximal one-run has length exactly $d$, the local complexity
measures can be computed even if the intervening zero-gaps are arbitrarily
long.  It also isolates the first source of additional sensitive blocks
within such gaps.

\begin{theorem}[Zero-gap decomposition and exact local complexities]
\label{thm:regular-run-zero-gaps}
Let $d\ge4$ be even and suppose a nonconstant binary word $x$ has the
cyclic run decomposition
\[
 x=1^d0^{g_1}1^d0^{g_2}\cdots1^d0^{g_r},
 \qquad r\ge1,\quad g_i\ge1.
\]
Let $n=rd+\sum_{i=1}^r g_i$ and $L_i=(g_i-2)_+$, where
$u_+=\max\{u,0\}$.  Then
\[
 \boxed{s(F_{d,n},x)=n-\sum_{i=1}^r|g_i-2|}
\]
and
\[
 \boxed{bs(F_{d,n},x)=C(F_{d,n},x)
 =s(F_{d,n},x)+\sum_{i=1}^r\left\lfloor\frac{L_i}{d}\right\rfloor.}
\]
In particular, $s(F_{d,n},x)=bs(F_{d,n},x)=C(F_{d,n},x)$
if and only if every $g_i\le d+1$.
\end{theorem}
\begin{proof}
All coordinate indices are cyclic.  Each one-coordinate lies in exactly
one all-one interval of length $d$, namely its entire maximal one-run.
It is therefore sensitive.  In a one-zero gap, flipping the unique zero
creates exactly $d$ all-one intervals, so the zero is insensitive because
$d$ is even.  In a two-zero gap each boundary zero is sensitive.  In a
gap of length $g\ge3$, exactly the first and last zeros are sensitive:
flipping either creates a unique all-one interval using the adjacent
run of $d$ ones, while flipping an interior zero leaves zeros on both
sides and creates none.  Hence the number of insensitive coordinates is
\[K(x)=n-s(F_{d,n},x)=\sum_i|g_i-2|.\]

Fix all sensitive coordinates at their values in $x$ and let the
insensitive coordinates vary.  A free coordinate in a one-zero gap is
separated from every other free coordinate in a different gap by a full
one-run of length $d$: no degree-$d$ interval contains both.  A degree-$d$
interval containing this free zero otherwise consists of fixed ones,
and there are exactly $d$ such intervals.  Thus all its linear
contributions cancel in $\mathbb F_2$.

For a gap of length $g\ge3$, the first and last zeros remain fixed to
zero.  Every interval meeting a free interior coordinate and extending
outside that interior encounters one of these fixed zeros and vanishes.
Therefore the only nonconstant terms on the $L=g-2$ interior variables
are the length-$d$ consecutive products entirely inside the interior.
Distinct interiors do not interact.  The restricted polynomial is thus
\[
 c\oplus\bigoplus_{i:g_i\ge3}P_{L_i,d}(z^{(i)}),\qquad
 P_{L,d}(z)=\bigoplus_{j=1}^{L-d+1}\prod_{a=j}^{j+d-1}z_a,
\]
where an empty sum means $P_{L,d}=0$ if $L<d$.

At $0^L$, a block is sensitive for $P_{L,d}$ precisely when the
coordinates changed to one contain an odd number of all-one length-$d$
intervals.  In particular, a sensitive block must contain all the
coordinates of at least one such interval, and so has size at least $d$.
This gives $bs(P_{L,d},0^L)\le\lfloor L/d\rfloor$.  Conversely,
$\lfloor L/d\rfloor$ disjoint intervals of length exactly $d$ are
sensitive blocks, because changing precisely one such interval from
zeros to ones makes exactly one summand nonzero.  Hence
\[bs(P_{L,d},0^L)=\lfloor L/d\rfloor.\]

A zero-certificate for $P_{L,d}$ at $0^L$ must fix at least one zero in
every consecutive length-$d$ interval: otherwise setting exactly one
unhit interval to ones makes the function equal to one.  Conversely,
a set meeting all these intervals forces every monomial to vanish.
The minimum size of such a hitting set is $\lfloor L/d\rfloor$:
disjoint length-$d$ intervals give the lower bound, and fixing the
positions $d,2d,\ldots,\lfloor L/d\rfloor d$ meets every such interval.
Thus $C(P_{L,d},0^L)=\lfloor L/d\rfloor$.

For completeness, if a restriction splits into a sum of functions on
disjoint sets of variables, each summand having value zero at the base
input, both local block sensitivity and local certificate complexity add.
For block sensitivity, a sensitive block must be sensitive in at least
one component, so its disjoint packing can be charged to a distinct
component; combining componentwise optimum packings gives equality.
For certificates, any unfixed variables on which a component can change
its value would prevent the whole sum from being constant (hold all other
components at their base inputs).  Thus each component requires its own
certificate, and componentwise certificates suffice.  The originally
sensitive coordinates must each be fixed in a certificate and provide
$s(F_{d,n},x)$ disjoint singleton sensitive blocks.  Applying these
observations to the restricted polynomial proves the asserted formulas.

Finally, $\lfloor(g_i-2)_+/d\rfloor=0$ exactly when $g_i\le d+1$,
proving the equivalence.
\end{proof}

\begin{remark}
The theorem concerns \emph{local} complexities at the specified input;
it does not by itself yield the corresponding global certificate
complexity.  For small-deficit extremizers having only one-zero and
two-zero gaps, it recovers their joint-freeness conclusion without a
separate restriction argument.
\end{remark}

\begin{theorem}[One exceptional long run with unrestricted zero-gaps]
\label{thm:exceptional-run-zero-gaps}
Let $d\ge4$ be even, let $1\le t\le(d-2)/2$, and suppose a
nonconstant cyclic word has the run decomposition
\[
 x=1^{d+2t}0^{g_1}1^d0^{g_2}\cdots1^d0^{g_r},
 \qquad r\ge1,\quad g_i\ge1.
\]
For $r=1$ the expression denotes one cyclic one-run followed by
one cyclic zero-gap.  Put
\[
 L_i=(g_i-2)_+,
 \qquad K=2t+\sum_{i=1}^r|g_i-2|,
 \qquad H=\sum_{i=1}^r\left\lfloor\frac{L_i}{d}\right\rfloor.
\]
Then
\[
 \boxed{s(F_{d,n},x)=n-K},\qquad
 \boxed{bs(F_{d,n},x)=C(F_{d,n},x)=n-K+t+H.}
\]
In particular, if $1\le g_i\le d+1$ for every $i$, then
$bs(F_{d,n},x)=C(F_{d,n},x)=s(F_{d,n},x)+t$.
\end{theorem}
\begin{proof}
Number the exceptional one-run from $0$ through $d+2t-1$.
The number of all-one length-$d$ monomials containing position $j$
inside this run is
\[
 a(j)=\min\{j+1,d+2t-j,2t+1,d\}.
\]
Since $2t+1<d$, its insensitive one-coordinates are precisely
\[
 U=\{1,3,\ldots,2t-1\}\cup
   \{d,d+2,\ldots,d+2t-2\}.
\]
Every one in a regular length-$d$ run is sensitive.  A zero-gap of
length one has an insensitive zero (flipping it creates $d$ new
all-one monomials); a gap of length two has no insensitive zeros;
a gap of length $g\ge3$ has exactly $g-2$ insensitive interior zeros.
The claimed sensitivity follows.

We first prove a lower bound for block sensitivity.  The $t$ disjoint
pairs $\{2j+1,d+2j\}$ ($0\le j<t$) lie in $U$ and are sensitive:
exactly one all-one length-$d$ monomial contains both their
coordinates, while each coordinate is individually insensitive.
For each gap with $L_i\ge d$, select $\lfloor L_i/d\rfloor$
disjoint length-$d$ intervals in its insensitive interior.
Flipping the coordinates of each interval creates exactly one
all-one monomial.  These blocks are disjoint from one another and
from the exceptional-run pairs.  Adding all sensitive singletons,
we obtain
\[
 bs(F_{d,n},x)\ge s(F_{d,n},x)+t+H.
\]

For the reverse certificate bound, fix all sensitive coordinates
at their values in $x$.  Among the $2t$ insensitive positions in
the exceptional run, fix the second half
$\{d,d+2,\ldots,d+2t-2\}$, and leave free
\[
 S=\{1,3,\ldots,2t-1\}.
\]
Leave the zero of every one-zero gap free.  In each gap of length
$g\ge3$, its two boundary zeros are already fixed to zero; in its
$L=g-2$ interior, fix a minimum set of
$\lfloor L/d\rfloor$ zeros meeting every length-$d$ interval.
Leave the other interior zeros free.  We show that the resulting
restriction of $F_{d,n}$ is constant.

No degree-$d$ monomial can involve free variables in two distinct
gap interiors or simultaneously involve a gap-interior free
variable and one from a one-run: it would pass through one of the
fixed boundary zeros.  No monomial consisting solely of interior
coordinates survives nontrivially, by the chosen hitting sets.
Two different one-zero gap variables never appear in one degree-$d$
monomial, since a full length-$d$ one-run separates them.
Likewise, a one-zero gap variable cannot share a degree-$d$
interval with a free gap-interior variable.

It remains to check monomials involving $S$ and one-zero gap
variables.  A one-zero gap not immediately preceding the
exceptional run is separated from $S$ by at least $d$ fixed
one-coordinates (or a fixed boundary zero), and so cannot share a
length-$d$ interval with an element of $S$.  The same statement
holds for a one-zero gap immediately following the exceptional run,
unless it is also the preceding gap in the cyclic case $r=1$;
that cyclic case is included in the following count.

Write $v$ for the free zero immediately preceding the exceptional
run, if that gap has length one.  For a nonempty set $T\subseteq S$,
the coefficient of the centered monomial $z_v\prod_{j\in T}z_j$
is the parity of the number of length-$d$ windows containing $v$
and all of $T$ with their remaining coordinates initially one.
These windows start at the $d-1-\max T$ consecutive allowable
positions immediately before $v$ (read cyclically).  Thus their
number is
\[
 d-1-\max T\equiv0\pmod2,
\]
where $\max T$ is odd and $\max T\le2t-1\le d-3$.
The coefficient of $z_v$ alone is zero as well: its original
coordinate is insensitive.  Hence all mixed terms involving $v$
and $S$ vanish.

Finally, without any one-zero gap variable, the coefficient of a
nonempty centered monomial $z_T$ ($T\subseteq S$) equals the parity
of the number of all-one length-$d$ windows of the exceptional
one-run containing $T$.  Their starts are precisely
$0,1,\ldots,\min T$; this is valid because
$\max T\le2t-1\le d-3$.  Their number $\min T+1$ is even.
Thus all nonconstant coefficients in the entire restricted
polynomial vanish, so the chosen free coordinates define a
certificate at $x$.

Among insensitive coordinates, the certificate fixes exactly $t$
exceptional-run positions and $H$ gap-interior positions.  Hence
$C(F_{d,n},x)\le s(F_{d,n},x)+t+H$.  Combining this with
$bs\le C$ and the sensitive-block lower bound proves the formulas.
\end{proof}

\begin{remark}
The theorem determines the three \emph{local} complexities at the
specified word; it does not determine their global maxima.
Unlike the preceding zero-gap theorem, its proof permits one-zero
gaps adjacent to the exceptional long run.  The certificate
argument handles their possible mixed coefficients by parity,
rather than asserting an independent-polynomial decomposition.
\end{remark}

\begin{lemma}[Guarded interval components]
\label{lem:guarded-interval-components}
Let $P_{T,d}(z)=\bigoplus_{j=0}^{T-d}\prod_{a=j}^{j+d-1}z_a$,
with $P_{T,d}=0$ for $T<d$.  Suppose $2\le T\le d+1$ and
$w=0u0$ is a word of length $T$.  If $T=d+1$, assume that $u$
contains a zero.  Then
\[
 bs(P_{T,d},w)=C(P_{T,d},w)=
 \begin{cases}0,&T<d,\\1,&d\le T\le d+1.\end{cases}
\]
\end{lemma}
\begin{proof}
If $T<d$ the polynomial is identically zero.  If $T=d$, it is
one product; its value at $0u0$ is zero, it is nonconstant, and
fixing an endpoint zero kills it.  If $T=d+1$, it consists of the
two length-$d$ products supported on $[0,d-1]$ and $[1,d]$.
They share the positions $1,\ldots,d-1$, which include a zero of
$u$ by hypothesis.  Fixing that zero kills both products, so
$C\le1$.  In either nontrivial case the polynomial is nonconstant:
the word supported on exactly the first length-$d$ interval makes
precisely that interval product nonzero.  Since the base value is
zero, a sensitive block exists.  Therefore $1\le bs\le C\le1$.
\end{proof}

\begin{theorem}[Mixed exceptional long runs and guarded short clusters]
\label{thm:mixed-guarded-clusters}
Let $d\ge4$ be even, $r\ge1$, and consider the cyclic word
\[
 x=\prod_{i=1}^{r}\bigl(1^{d+2t_i}\,00\,u_i\,00\bigr),
 \qquad 0\le t_i\le\frac{d-2}{2}.
\]
Put $T_i=|u_i|+2$ and assume $2\le T_i\le d+1$;
if $T_i=d+1$, assume $u_i$ contains a zero.  Set
\[
 a=\sum_i t_i,\qquad B=\sum_i T_i,\qquad
 M=\#\{i:T_i\in\{d,d+1\}\}.
\]
Then the exact local complexities are
\[
 \boxed{s(F_{d,n},x)=n-2a-B},\qquad
 \boxed{bs(F_{d,n},x)=C(F_{d,n},x)=n-a-B+M.}
\]
The statement includes $r=1$ and empty words $u_i$.
\end{theorem}
\begin{proof}
The two zeros of each displayed $00$ bordering a long run and a
cluster have different roles.  The zero adjacent to the long run
is sensitive: flipping it creates exactly one all-one length-$d$
window using the adjacent run.  The other zero belongs to the
segment $0u_i0$ and is insensitive.  All positions of that segment
are insensitive: its length is at most $d+1$ and its two endpoints
are zero, so flipping a single position cannot create an all-one
length-$d$ interval (for $T_i=d+1$ the additional internal zero
ensures this also for an interior-zero flip).

In the long run $1^{d+2t_i}$, the standard interval-incidence
parity count gives precisely $2t_i$ insensitive one-coordinates;
there are no other insensitive coordinates.  Therefore the
sensitivity deficit is $2a+B$.

After fixing all sensitive coordinates, each segment $0u_i0$ is
bounded by two sensitive zeros fixed to zero.  Every interval
crossing such a boundary vanishes.  Consequently the restricted
polynomial decomposes, up to a constant, as a sum on disjoint
variable sets: one exceptional-run component for each $t_i$ and
one guarded interval component $P_{T_i,d}$ for each $i$.
The exceptional-run calculation in the proof of
Theorem~\ref{thm:exceptional-run-zero-gaps} gives local block
sensitivity and certificate complexity $t_i$ for its centered
insensitive variables.  Lemma~\ref{lem:guarded-interval-components}
gives contributions $0$ or $1$ for each short cluster, with their
sum equal to $M$.  Both complexities add over disjoint independent
components at their prescribed base inputs: disjoint sensitive
blocks can be assigned to the component in which their flip changes
the value, and a certificate must independently make every
component constant.  Adding the $n-2a-B$ sensitive singletons
proves the two formulas.
\end{proof}

\begin{corollary}[Arbitrary isolated short-run clusters]
\label{cor:short-cluster-local}
Let $d\ge4$ be even and $r\ge1$. For the cyclic word
$x=\prod_{i=1}^r(1^d00\,u_i00)$, put $T_i=|u_i|+2$ and assume
$2\le T_i\le d+1$ for every $i$; when $T_i=d+1$, assume $u_i$
contains a zero. Then
\[
 s(F_{d,n},x)=n-\sum_iT_i,\qquad
 bs(F_{d,n},x)=C(F_{d,n},x)
 =n-\sum_iT_i+\#\{i:T_i\in\{d,d+1\}\}.
\]
\end{corollary}
\begin{proof}
Apply Theorem~\ref{thm:mixed-guarded-clusters} with every exceptional
long-run extension parameter equal to zero.  Its long-run contribution
to the insensitive-coordinate count then vanishes.  The remaining
components are the $r$ guarded short-run intervals of lengths $T_i$;
by Lemma~\ref{lem:guarded-interval-components}, each contributes
one additional disjoint sensitive block (and one additional required
certificate coordinate) precisely if $T_i=d$ or $d+1$.
All $n-\sum_i T_i$ coordinates outside the guarded intervals are
sensitive singleton positions.  The asserted formulas follow by
adding the disjoint component contributions as in the preceding theorem.
\end{proof}

\subsection{Complete block-sensitivity classification in every even degree}
\label{subsec:complete-even-bs}

The small-deficit charging lemma was established before the even-degree
classification, so the following theorem uses no result proved later.

\begin{theorem}[Global block sensitivity of consecutive even-degree monomials]
\label{thm:complete-even-bs}
Let $d=2k\ge4$, put $m=d+2$, and write
$n=qm+\rho$ with $0\le\rho<m$. Suppose
\[
q\ge \max\{2,k-1\}.
\]
Then
\begin{equation}\label{eq:complete-even-bs}
\boxed{\displaystyle
bs(F_{d,n})=n-\min\left\{\left\lceil\frac{\rho}{2}\right\rceil,
                                      m-\rho\right\}.}
\end{equation}
The corresponding global sensitivity is
$s(F_{d,n})=n-\min\{\rho,m-\rho\}$; its upper bound
is Lemma~\ref{lem:intrinsic-sensitivity-upper}, and matching
witness counts are recorded in Corollary~\ref{cor:intrinsic-sensitivity}.
\end{theorem}
\begin{proof}
Set $\mu=\min\{\lceil\rho/2\rceil,m-\rho\}$.
For $\rho=0$, the regular cyclic word $(1^d00)^q$ has all $n$
coordinates sensitive, proving the assertion. Assume $1\le\rho<m$,
and put $h=m-\rho$.

We first prove $bs(F_{d,n},x)\le n-\mu$ for an arbitrary input $x$.
The constant words have zero sensitivity and at most $n/2$ disjoint
sensitive blocks: every singleton is insensitive. Since
$\mu\le k$ and $n\ge2d-1$, this gives the required bound.
For a nonconstant word, let $r$ be the number of maximal one-runs
of length at least $d$, and set $K=n-s(F_{d,n},x)$.

\emph{Case 1: $r\ge q+1$.}
Between consecutive long runs the cyclic word has exactly $r$
units, each a plain zero-gap or a cluster containing short one-runs.
Define the nonnegative quantities
\[
A=\sum_{\rm long\ runs}(L-d),\qquad
E=\sum_{\rm clusters}(T-2),\qquad
P=\sum_{\rm plain\ gaps}\max(g-2,0).
\]
Let $N_1$ denote the number of plain zero-gaps of length one.
Counting lengths gives the exact identity
\[
 n=r(d+2)+A+E+P-N_1,
\quad\text{hence}\quad
N_1=r(d+2)-n+A+E+P\ge m-\rho=h.
\]
For clarity, the integer $N_1$ counts \emph{gaps}, not the total
number of zero coordinates. Each counted gap has exactly one zero;
its neighboring long runs prevent any degree-$d$ monomial from
involving two of these selected coordinates. Thus the restriction
of $F_{d,n}$ to the corresponding $N_1$-dimensional subcube has
degree at most one.
Each of these $N_1$ zero coordinates is individually insensitive:
flipping it creates exactly $d$ all-one degree-$d$ windows, whose
parity is zero. Any two distinct such zeros are separated in
both cyclic directions by a long one-run of length at least $d$.
Thus no degree-$d$ window contains two of them.
After fixing all remaining coordinates at $x$, the polynomial
restricted to these $N_1$ zeros is affine. All its linear
coefficients vanish by individual insensitivity, so the restricted
function is constant. We have a certificate fixing at most
$n-N_1$ coordinates, and hence
\[
bs(F_{d,n},x)\le C(F_{d,n},x)\le n-N_1\le n-h\le n-\mu.
\]

\emph{Case 2: $r\le q$.}
If $r=0$, no all-one cyclic interval of length $d$ occurs at $x$.
Every coordinate carrying a one is then insensitive, since changing
that one to zero cannot destroy an existing all-one degree-$d$
monomial and cannot create one. If at least one zero coordinate is
sensitive, changing that zero to one must create an odd (hence
nonzero) number of degree-$d$ monomials. One such monomial contains
the flipped zero together with $d-1$ positions that were already one.
Those $d-1$ positions are insensitive by the preceding observation,
so $K\ge d-1$. If no zero is sensitive, all coordinates are
insensitive and $K=n\ge d-1$. This also covers the all-zero word,
already considered separately.
If $K\ge d-1=2k-1$, the general inequality
\[
bs(f,x)\le\left\lfloor\frac{n+s(f,x)}2\right\rfloor
        =n-\left\lceil\frac{K}{2}\right\rceil
\]
gives $bs(F_{d,n},x)\le n-k\le n-\mu$, since $\mu\le k$.
Otherwise $K\le d-2$, and Lemma~\ref{lem:small-deficit-charge}
gives
\[
K\ge |n-rm|\ge n-rm\ge n-qm=\rho.
\]
The same general inequality yields
$bs(F_{d,n},x)\le n-\lceil\rho/2\rceil\le n-\mu$.
This completes the upper bound.

For the matching lower bound, first consider $\rho\le m/2$.
The regular-run and exceptional-long-run constructions in
Theorem~\ref{thm:lower-half-bs} attain
$bs=n-\lceil\rho/2\rceil$.
For $\rho>m/2$, use $h=m-\rho$.
One construction starts with $q+1$ one-runs of length $d$,
separated by exactly $h$ one-zero gaps and the remaining
$(q+1-h)$ gaps of length two. This exists because
$q+1\ge k\ge h$ and has $n-h$ sensitive singletons.
A second construction applies when
$\lceil\rho/2\rceil<h$. Then $\rho\le 2h-2\le d-2$.
For $\rho=2t$, take
\[
1^{d+2t}00(1^d00)^{q-1},
\]
and for $\rho=2t+1$, take
\[
1^{d+2t}000(1^d00)^{q-1}.
\]
In either case the long run contains precisely $2t$
insensitive coordinates and $t$ disjoint sensitive pairs
$\{2j+1,d+2j\}$, $0\le j<t$; in the odd-residue
case the middle zero of the three-zero gap is also insensitive.
The remaining sensitive singletons give
$bs\ge n-\lceil\rho/2\rceil$.
The first construction attains the required lower bound when
$h\le\lceil\rho/2\rceil$, and the second when the reverse
strict inequality holds. Thus the stated equality follows.
\end{proof}

\begin{corollary}[Self-contained global sensitivity formula]
\label{cor:intrinsic-sensitivity}
Under the hypotheses of Theorem~\ref{thm:complete-even-bs},
\[
\boxed{s(F_{d,n})=n-\min\{\rho,d+2-\rho\}.}
\]
\end{corollary}
\noindent\emph{Related work.} A corresponding even-degree sensitivity formula is also treated in \cite{LiZhangRSBFSensitivity}; the present paper includes an independent proof for the range stated here.

\begin{proof}
Put $m=d+2$ and $h=m-\rho$.
The upper bound $s(F_{d,n})\le n-\min\{\rho,h\}$
follows from Lemma~\ref{lem:intrinsic-sensitivity-upper},
which was proved independently before the lower-half classification.
If $\rho=0$, the regular word $(1^d00)^q$ attains sensitivity $n$.
Assume $1\le\rho<m$.
For the reverse bound when $\rho\le h$, first consider $\rho=2t$.
The word
\[
1^{d+2t}00(1^d00)^{q-1}
\]
has exactly $2t$ insensitive one-coordinates in the exceptional run;
all remaining coordinates are sensitive. Indeed, the exceptional-run
incidence sequence has even entries exactly at offsets
$1,3,\ldots,2t-1$ and $d,d+2,\ldots,d+2t-2$.
For $\rho=2t+1\le h$, the word
\[
1^{d+2t}000(1^d00)^{q-1}
\]
has those same $2t$ insensitive one-coordinates and precisely one
additional insensitive zero, namely the middle zero of the three-zero
gap. In each case all length-$d$ ordinary one-run positions and the
remaining gap boundary zeros are sensitive. The direct counts give
$s(F_{d,n})=n-\rho$ at these witnesses, without using the conclusions
of the preceding hierarchy or lower-half classification.
If $h<\rho$, the $q+1$ long-run
word with $h$ one-zero gaps used in the lower-bound proof of
Theorem~\ref{thm:complete-even-bs} has exactly $h$ insensitive
coordinates: every position of its length-$d$ one-runs is sensitive,
the zeros of its length-two gaps are sensitive, and its plain one-zero
gap zeros are insensitive because $d$ is even. Thus it has sensitivity
$n-h$. The asserted equality follows.
\end{proof}

\begin{remark}[Scope and independence]
The proof of the block-sensitivity formula uses the self-contained
small-deficit lemma and explicit lower-bound words. The maximum-sensitivity formula is now also derived within
the present manuscript in Corollary~\ref{cor:intrinsic-sensitivity}. The result alone does not imply an exact formula for
global certificate complexity; $bs(F_{d,n})\le C(F_{d,n})$ remains
our general bound there.
\end{remark}

\begin{corollary}[Upper-half residue consequences]
\label{cor:all-upperhalf}
Let $d=2k\ge4$, $1\le h\le k$, and
$n=(q+1)(d+2)-h$, under the length hypothesis of
Theorem~\ref{thm:complete-even-bs}. Then
\[
bs(F_{d,n})=n-\min\left\{h,
          \left\lceil\frac{d+2-h}{2}\right\rceil\right\}.
\]
In particular the previously established deficit-two through
deficit-six and diagonal block-sensitivity formulas are instances
of one closed formula; their individual structural rigidity
statements are separate refinements, not needed for this proof.
\end{corollary}
\begin{proof}
Write $m=d+2$.  Since $n=(q+1)m-h=qm+(m-h)$, the residue
in Theorem~\ref{thm:complete-even-bs} is $\rho=m-h$.
Substitution into that theorem gives
\[
bs(F_{d,n})=n-\min\left\{
  \left\lceil\frac{m-h}{2}\right\rceil,\,m-(m-h)
\right\},
\]
which is the displayed formula.  The references to particular
deficit cases are specializations of this identity, not additional
assertions requiring new hypotheses.
\end{proof}

\section{A general two-free-coordinate certificate theorem}
\label{sec:pair-certificate}
This section gives a degree-independent certificate upper bound for
consecutive even-degree monomials. Recall the Boolean derivatives
$D_if$ and $D_iD_jf$ defined in the preliminaries.
Two individually insensitive coordinates $i,j$ can be left free
simultaneously, fixing all other coordinates at $x$, precisely when
$D_iD_j f(x)=0$.

\begin{lemma}[Triangle-free local mixed derivatives]
\label{lem:three-derivative}
Let $d\ge4$ be even and $n\ge3d$.  If distinct coordinates
$i,j,k$ lie in a common cyclic interval of length $d$ and
$D_iF_{d,n}(x)=D_jF_{d,n}(x)=D_kF_{d,n}(x)=0$, then at least one
of $D_iD_jF_{d,n}(x)$, $D_iD_kF_{d,n}(x)$, and
$D_jD_kF_{d,n}(x)$ vanishes.
\end{lemma}
\begin{proof}
Rotate the coordinates and use integer positions
$0<p<\delta<d$ for the triple.  The assumption $n\ge3d$
ensures that all relevant length-$d$ windows can be treated as
ordinary intervals in this local integer indexing.  Suppose for a
contradiction that all three mixed derivatives equal one.
Since $D_0D_\delta F(x)=1$, some length-$d$ window containing
$0$ and $\delta$ has all its other coordinates equal to one.
Thus $x_1=\cdots=x_{\delta-1}=1$.

First, if $x_0=x_\delta=1$, the three coordinates lie in one
maximal one-run, except when $x=1^n$.  In that exceptional case,
every coordinate is insensitive because $d$ is even.  Directly
counting the length-$d$ windows containing each pair gives
\[
D_0D_pF(1^n)\equiv d-p\equiv p,\qquad
D_pD_\delta F(1^n)\equiv d-(\delta-p)\equiv\delta-p,
\]
and
\[
D_0D_\delta F(1^n)\equiv d-\delta\equiv\delta
\pmod 2.
\]
These three parities cannot all be one, since
$p+(\delta-p)=\delta$.  We may therefore assume that the
maximal one-run containing the three coordinates is bounded by
zeros.  We record the necessary parity fact.  For a run
of length $L\ge d$, its all-one windows have starting positions
$s=0,\ldots,L-d$.  Set
\[
 P(a)=\#\{s:0\le s\le L-d,\ s\le a\}\pmod 2,
 \qquad P(a)=0\quad(a<0).
\]
For a position $u$ in that run, $D_uF=P(u)\oplus P(u-d)$.
For positions $u<v$ in the run at distance less than $d$,
$D_uD_vF=P(u)\oplus P(v-d)$.
When $u,v$ are insensitive, this equals $P(u)\oplus P(v)$;
for pairs separated by at least $d$, the mixed derivative is zero.
Hence the mixed-derivative graph on insensitive one-coordinates
inside a single run is bipartite according to the value of $P$.
It cannot contain a triangle.

Next, suppose precisely one endpoint is zero.  Reflecting if
necessary, take $x_0=0$, $x_\delta=1$.  Then $p$ and $\delta$
belong to the one-run starting at position $1$.  Every all-one
length-$d$ window containing $p$ also contains $\delta$:
its start is at least $1$ and at most $p$, and its right end is
at least $d\ge\delta$.  Thus $D_pD_\delta F=D_pF=0$,
a contradiction.

Finally, suppose $x_0=x_\delta=0$.  Let $a$ and $b$ denote
the lengths of the consecutive one-runs immediately to the left of
$0$ and to the right of $\delta$, truncated at $d-1$.
The middle coordinates $1,\ldots,\delta-1$ are all one.
The numbers of contributing windows to the first derivatives at
$0$ and $\delta$ are
\[
 N_0=\max\{0,a+\delta-d+1\},\qquad
 N_\delta=\max\{0,b+\delta-d+1\}.
\]
Both numbers are even.  Counting starts for windows containing
both $0$ and $p$ gives
\[
 N_{0p}=\max\left\{0,\ \delta-d-
                 \max(-a,p-d+1)+1\right\}.
\]
This number is odd.  If the inner maximum were $-a$, then
$N_{0p}=N_0$, a contradiction.  Hence
$p-d+1>-a$, so $a\ge d-p$, and
\[
 N_{0p}=\delta-p\equiv1\pmod2.
\]
Likewise, put $v=\delta+b-d+1$.  Windows containing $\delta$
have starts from $1$ through $v$ when $v>0$, while those
containing both $p$ and $\delta$ have starts from $1$ through
$\min(p,v)$.  Their counts are $\max(0,v)$ (even) and
$\max(0,\min(p,v))$ (odd).  Therefore $v>p$, and hence
\[
 p\equiv1\pmod2,\qquad b\ge d-\delta+p.
\]
It follows that $\delta$ is even.

Every length-$d$ window containing both $0$ and $\delta$
starts at one of $\delta-d+1,\ldots,0$.  The exterior-run bounds
$a\ge d-p$ and $b\ge d-\delta+p$ ensure that all other
positions in each such window are one.  The mixed derivative
therefore counts exactly $d-\delta$ windows, which is even
because $d$ and $\delta$ are even.  This contradicts
$D_0D_\delta F(x)=1$ and completes the proof.
\end{proof}

\begin{lemma}[Cyclic localization]
\label{lem:cyclic-localization}
Suppose $n\ge3d$ and $S\subseteq\mathbb Z_n$ has cyclic
pairwise distance less than $d$.  Then $S$ is contained in a
cyclic interval of $d$ consecutive positions.
\end{lemma}
\begin{proof}
If $S=\varnothing$, the assertion is immediate. Otherwise, fix
$v\in S$ and rotate so that $v=0$.  Every other member
lies either at a positive distance $a\in\{1,\ldots,d-1\}$
or at a negative distance $-b$, $b\in\{1,\ldots,d-1\}$.
If members occur on both sides, choose the largest such positive
and negative distances $a,b$.  Their cyclic distance is
$\min(a+b,n-a-b)$.  As $a+b\le2d-2$ and $n\ge3d$,
$n-a-b\ge d+2$.  The pairwise-distance hypothesis therefore
forces $a+b\le d-1$.  The arc from $-b$ to $a$ contains all
members of $S$ and has at most $d$ positions.  The one-sided
case is immediate.
\end{proof}

\begin{theorem}[Two-free-coordinate certificate bound]
\label{thm:two-free-certificate}
Let $d\ge4$ be even and $n\ge3d$.  If
$s(F_{d,n})\le n-3$, then
\[
 \boxed{C(F_{d,n})\le n-2.}
\]
\end{theorem}
\begin{proof}
Every input $x$ has at least three insensitive coordinates.
If two have cyclic distance at least $d$, no length-$d$ monomial
contains both; since both first derivatives vanish, fixing the
other $n-2$ coordinates gives a certificate.  Otherwise the
insensitive coordinates all have pairwise cyclic distance less
than $d$.  Lemma~\ref{lem:cyclic-localization} places them
within one cyclic interval of length $d$.  Choose any three.
Lemma~\ref{lem:three-derivative} supplies a pair with vanishing
mixed derivative; their individual derivatives already vanish.
They can therefore remain jointly free in a certificate.
\end{proof}

\begin{corollary}[Exact certificate complexity in residue $d$]
\label{cor:even-residue-d-cert}
Let $d\ge4$ be even and $n=q(d+2)+d$, with
$q\ge\max\{2,d/2-1\}$. Then
\[
\boxed{s(F_{d,n})=bs(F_{d,n})=C(F_{d,n})=n-2.}
\]
\end{corollary}
\begin{proof}
Theorem~\ref{thm:complete-even-bs} and
Corollary~\ref{cor:intrinsic-sensitivity} give $bs=s=n-2$.
The length hypothesis implies $n\ge3d$. Fix any input $x$, and
put $K=n-s(F_{d,n},x)$. If $K\ge3$, the pointwise construction
in the proof of Theorem~\ref{thm:two-free-certificate} supplies two
jointly free coordinates. If $K=2$, the small-deficit charging lemma
applies. In its notation $n=(q+1)(d+2)-2$ and
$K\ge |n-r(d+2)|$ force $r=q+1$ (as $d+2>2$).
The length identity $n=r(d+2)+A+B+E$ and the charge
$K\ge A+\sum|g-2|+E$ then force
$A=E=0$ and exactly two plain one-zero gaps, all other plain gaps
having length two. Indeed, $A+E+B=-2$ and
$A+E+\sum|g-2|\le2$, while $B\ge-\sum|g-2|$;
thus equality requires $A=E=0$ and precisely two negative
unit deviations. Their zeros are jointly free by Case~1 of
Theorem~\ref{thm:complete-even-bs}. Finally $K=0,1$ cannot occur,
since global sensitivity is $n-2$. Hence $C(F_{d,n},x)\le n-2$
for every input, and $bs\le C$ gives equality.
\end{proof}

\begin{corollary}[Exact global certificate complexity in residue three]
\label{cor:even-residue-three-cert}
Let $d\ge4$ be even and $n=q(d+2)+3$, with
$q\ge\max\{2,d/2-1\}$ and $n\ge3d$.  Then
\[
 \boxed{bs(F_{d,n})=C(F_{d,n})=n-2.}
\]
\end{corollary}
\begin{proof}
The complete block-sensitivity formula,
Theorem~\ref{thm:complete-even-bs}, gives $bs(F_{d,n})=n-2$.
Corollary~\ref{cor:intrinsic-sensitivity} gives
$s(F_{d,n})=n-3$ under these hypotheses.  Applying
Theorem~\ref{thm:two-free-certificate} and
$bs\le C$ finishes the proof.
\end{proof}

\begin{remark}
The preceding argument determines global certificate complexity
in this residue class but does not assert $C(F_{d,n})=bs(F_{d,n})$
for every residue.  More than two jointly free coordinates
require control of higher-order Boolean derivatives.
\end{remark}

\section{Exact Local Complexity at the All-One Input in Odd Length}
\label{sec:odd-local-gap}
The global block-sensitivity hierarchy of the preceding section also
motivates the study of differences between complexity measures at a
\emph{fixed input}. These local differences need not equal the differences
between global maxima.

\begin{theorem}[An exact odd-length local certificate gap]
\label{thm:odd-local-gap}
Let $d\ge4$ be even and let $n\ge2d-1$ be odd. Then
\[
\boxed{
 s(F_{d,n},1^n)=0,\qquad
 bs(F_{d,n},1^n)=\frac{n-1}{2},\qquad
 C(F_{d,n},1^n)=\frac{n+d-1}{2}.}
\]
Consequently,
\[
 C(F_{d,n},1^n)-bs(F_{d,n},1^n)=\frac d2.
\]
\end{theorem}
\begin{proof}
Every coordinate belongs to exactly $d$ distinct cyclic length-$d$
monomials, so its first Boolean derivative at $1^n$ is zero.
Thus $s(F_{d,n},1^n)=0$, and every sensitive block has at least
 two elements. It follows that $bs(F_{d,n},1^n)\le (n-1)/2$.
An adjacent pair $\{i,i+1\}$ belongs jointly to exactly $d-1$
length-$d$ intervals (the other cyclic separation is at least $d$).
Because $d-1$ is odd and both first derivatives vanish, flipping the
pair changes the function. Choosing $(n-1)/2$ disjoint adjacent pairs
establishes the block-sensitivity formula.

Suppose a set $S$ of coordinates can all be left free in a
certificate at $1^n$, with every coordinate outside $S$ fixed to one.
For distinct $i,j\in S$, the mixed derivative
$D_iD_jF_{d,n}(1^n)$ must vanish. If their clockwise separation
is an odd integer $h$ with $1\le h<d$, their other separation is
$n-h\ge d$; exactly $d-h$ length-$d$ cyclic intervals contain
both positions. As $d-h$ is odd, their mixed derivative is one,
a contradiction. The same restriction applies in the opposite
clockwise direction. Hence every odd cyclic gap between
consecutive members of $S$ is at least $d+1$.

If $k=|S|\ge2$, its $k$ positive cyclic gaps sum to the odd
integer $n$, so at least one gap is odd and therefore at least
$d+1$. Every other gap is at least two, since a gap of one
is forbidden. Consequently
\[
 n\ge(d+1)+2(k-1),\qquad k\le\frac{n-d+1}{2}.
\]
For $k\le1$ the bound follows from $n\ge2d-1$. Thus every
certificate fixes at least $(n+d-1)/2$ coordinates.

Conversely, put $k=(n-d+1)/2$ and leave free the set
\[
 S=\{0,2,4,\ldots,2(k-1)\}=\{0,2,\ldots,n-d-1\}.
\]
All coordinates outside $S$ are fixed to one. The cyclic gaps
between consecutive members of $S$ are two, apart from a
closing gap of length $d+1$. We prove that the restriction is
constant, including all higher-order interactions among free
coordinates.

Write $z_i=1+x_i$ for $i\in S$, taking coefficients over
$\mathbb F_2$. In the expansion
\[
 \prod_{j\in I}x_j=\prod_{j\in I\cap S}(1+z_j)
   =\sum_{T\subseteq I\cap S}\prod_{j\in T}z_j,
\]
the coefficient of a nonempty squarefree monomial
$\prod_{j\in T}z_j$ in the restricted function equals the
parity of the number of cyclic length-$d$ intervals $I$
containing $T$. For $|T|=1$, that count is $d$, which is even.
For $|T|\ge2$, no length-$d$ interval containing two points of
$S$ can cross the closing gap of length $d+1$.
Set $a=\min T$ and $b=\max T$ in the displayed linear ordering.
The number of such intervals is
\[
 \max\{0,d-(b-a)\}.
\]
Indeed, if $b-a<d$, their starting positions are exactly
$b-d+1,\ldots,a$; otherwise there are none. As $a,b$ are even
and $d$ is even, this number is even. All nonconstant
coefficients in the $z$ variables therefore vanish. The
restricted function is constant, showing
$C(F_{d,n},1^n)\le n-k=(n+d-1)/2$ and proving equality.
\end{proof}

\begin{remark}
The difference $d/2$ is unbounded when the even degree varies,
but the theorem concerns \emph{local} values at $1^n$; it
makes no assertion that $C(F_{d,n})-bs(F_{d,n})=d/2$.
The use of centered variables $z_i=1+x_i$ in the proof is
essential: coefficients in the original free variables $x_i$
count intervals whose intersection with $S$ is exactly a
specified set, rather than merely containing that set.
\end{remark}

\section{A sextic global certificate family}\label{sec:sextic-upper-deficit-three}
This section establishes an exact global certificate value for the
sextic residue class $n\equiv5\pmod8$. We give the run-and-gap
certificate construction needed for the argument, rather than assuming
a local certificate identity for arbitrary mixed configurations.

\begin{lemma}[Run-and-gap free-coordinate bound]\label{lem:sextic-run-gap}
Suppose $x$ is nonconstant and has $r\ge1$ maximal one-runs of
length at least six. Write $w$ for the total number of ones in the
shorter one-runs, $g$ for the lengths of all maximal zero-gaps, and
$u(R)$ for the number of insensitive one-coordinates in a long run $R$.
Set
\[
 a=\sum_{R:\,|R|\ge6}\left\lceil\frac{u(R)}2\right\rceil,
 \qquad b(g)=(g-2)_+-\left\lfloor\frac{(g-2)_+}{6}\right\rfloor,
 \qquad B=\sum_g b(g),\quad S=a+w+B.
\]
Then
\[
 C(F_{6,n},x)\le n-S,\qquad n-8r\le2a+3w+2B.
\]
\end{lemma}
\begin{proof}
For the certificate bound, fix every zero except selected interior
zeros of gaps of length at least three. In a gap of length $g\ge3$,
keep its first and last zeros fixed and fix every sixth of its
$g-2$ interior positions, numbering these positions from one. Every
length-six interval contained in the interior meets a fixed zero;
any interval crossing its boundary meets the fixed outer zero.
Consequently $b(g)$ of its interior coordinates may be left free.
Every one-coordinate in a short one-run may be left free. For a long
one-run, with all zeros fixed, the prefix-parity formula for interval monomials partitions its
insensitive coordinates into two jointly free classes, as follows.
Number the positions in a run of length $L$ by $0,\ldots,L-1$,
and put
\[
 P(a)=\#\{s:0\le s\le L-6,\ s\le a\}\pmod2,
 \qquad P(a)=0\quad(a<0).
\]
Write the run polynomial in centered variables $z_i=1+x_i$.
The coefficient of a nonempty product supported on a set $T$
of free run coordinates is the parity of the length-six intervals
containing $T$. If $\max T-\min T<6$, this coefficient is
$P(\min T)+P(\max T-6)$ in $\mathbb F_2$; if the span is
at least six, the coefficient is zero. For a singleton $T=\{v\}$,
insensitivity says $P(v)=P(v-6)$. Thus partition insensitive
positions according to $P(v)\in\{0,1\}$. For any nonempty set
$T$ within either class and of span less than six, its coefficient
is $P(\min T)+P(\max T)=0$; for larger span it is already
zero. Every nonconstant coefficient of the restricted polynomial
therefore vanishes, proving joint freeness, not merely pairwise
insensitivity. Choosing the larger class leaves at least
$\lceil u(R)/2\rceil$ coordinates free. Fixed gap-boundary zeros
prevent cross-component monomials, proving the first assertion.

For the charge, a long run of length $L\ge6$ has at least $L-6$
insensitive ones, and hence $L-6\le2\lceil u(R)/2\rceil$.
Each gap length satisfies $g\le2+2b(g)$, and also $g-2\le2b(g)$.
Partition the shorter one-runs into maximal clusters between long
runs. A cluster with $j$ short one-runs, $v$ ones in total, and
$j+1$ zero-gaps has excess length, relative to one baseline
length-two gap, equal to
\[
 E=v+\sum_{\ell=0}^{j}g_\ell-2
 \le v+2j+2\sum_{\ell=0}^{j}b(g_\ell)
 \le3v+2\sum_{\ell=0}^{j}b(g_\ell).
\]
A direct gap between long runs contributes excess $g-2\le2b(g)$.
Summing the disjoint contributions gives $n-8r\le2a+3w+2B$.
\end{proof}

\begin{theorem}[Exact sextic certificate complexity in residue five]
\label{thm:sextic-residue-five}
For every integer $q\ge2$, putting $n=8q+5$, one has
\[
 \boxed{s(F_{6,n})=bs(F_{6,n})=C(F_{6,n})=n-3.}
\]
\end{theorem}
\begin{proof}
We first show $s(F_{6,n})=n-3$ independently of the global
block-sensitivity classification. Write $K(x)=n-s(F_{6,n},x)$.
If $K(x)\le2$, Lemma~\ref{lem:small-deficit-charge} (with $d=6$)
would give $K(x)\ge|n-8r|\ge3$, since $n\equiv5\pmod8$;
this is impossible. Thus $s(F_{6,n})\le n-3$.
For the reverse inequality, take the cyclic word
\[
 x=(1^60)^3(1^600)^{q-2}.
\]
Its three gaps of length one contribute exactly three insensitive
zeros, while all six ones in each length-six run and both boundary
zeros in each length-two gap are sensitive. Hence $K(x)=3$ and
$s(F_{6,n})=n-3$. We now seek three
coordinates that can remain jointly free at each input.

The all-zero and all-one inputs admit three pairwise cyclically
six-separated insensitive coordinates: for $0^n$ this is immediate
since any one-coordinate flip cannot create an all-one sextic
interval, and for $1^n$ each flip changes six monomials, an even
number. For an input with no length-six one-run, leave every one
free if there are at least three; otherwise it has at most two
ones, so at least $n-2\ge19$ individually insensitive zeros,
from which three cyclically six-separated positions may be chosen.
In either case a constant three-dimensional subcube results,
for two different reasons. When there are at least three ones,
free any three of them: setting any subset of those ones to zero
cannot create an all-one sextic interval. When at most two ones
occur, choose three mutually six-separated zeros. No sextic
monomial contains two of these free coordinates, so the restriction
to the three-dimensional subcube is affine; its linear coefficients
vanish because each chosen zero is individually insensitive.
Equivalently, no joint reassignment can create an all-one
sextic interval.

Now assume that there are $r\ge1$ long one-runs. When $r\ge q+1$,
the unconditional length identity in
Lemma~\ref{lem:unconditional-run-length} forces at least
$8r-n\ge3$ direct gaps of length one between long runs.
Each such zero is insensitive (the six created monomials cancel),
and no length-six monomial contains two such zeros. Hence they
are jointly free.

When $1\le r\le q$, use Lemma~\ref{lem:sextic-run-gap}.
If $S\ge3$, the assertion follows. Otherwise
\[
 5+8(q-r)=n-8r\le2S+w\le3S\le6.
\]
It follows that $r=q$, $S=2$, and $w\in\{1,2\}$.
We examine these two cases, keeping track of the slack in the
charging inequalities.

If $w=1$, then $a+B=1$ and
$5=n-8q\le2(a+B)+3w=5$. Thus every charge is tight.
The single short run is an isolated one, both adjoining gaps have
length two, all direct gaps have length two, and $B=0$.
The total excess of long-run lengths is two. It cannot consist of
two length-seven runs, because a length-seven run has five
insensitive one-coordinates and would contribute three to $a$.
It consists of one length-eight run and $q-1$ length-six
runs. Here is an explicit calculation of its insensitive positions.
Number the eight positions of the exceptional run $0,1,\ldots,7$.
The only all-one sextic intervals contained in this run are
$[0,5]$, $[1,6]$, and $[2,7]$. Their incidence counts at the
eight positions are
\[
 (1,2,3,3,3,3,2,1).
\]
Since every other sextic interval contains an original zero,
the insensitive one-positions are exactly $1$ and $6$.
Leave position $1$ of this run free.

For the isolated short one, choose either of its adjacent
length-two gaps and write the resulting local word as
\[
 \cdots 1^{L}\,\underbrace{0_{\mathrm{fixed}}
 0_{\mathrm{free}}}_{\text{chosen gap}}
 1_{\mathrm{free}}\,
 \underbrace{0_{\mathrm{fixed}}0_{\mathrm{fixed}}}
 _{\text{other gap}}\,1^{L'}\cdots,
 \qquad L,L'\ge6.
\]
All zeros except the displayed free zero are fixed.
Any sextic interval involving either of the two free variables
near the isolated one contains one of the displayed fixed zeros:
within the central segment, at most two consecutive positions can
become one, and a sextic interval leaving it crosses a fixed zero.
Moreover, the free coordinate in the exceptional length-eight run
is separated from this segment by fixed zeros, even if that run
is one of the neighboring $1^{L}$ or $1^{L'}$ runs. Its derivative
is zero by the incidence calculation, and no sextic interval can
contain variables from both components. Thus the three free
coordinates form a constant subcube.

If $w=2$, then $a=B=0$, so every long run has length six and every
gap has length one or two. A single short run of length two would
contribute at most four to $n-8q$, which is impossible since this
excess is five. Thus there are precisely two isolated one-runs.
Their two cluster excess contributions sum to at most six; to
obtain exactly five, exactly one of the gaps has length one, and
all other gaps have length two. Leave both isolated one-coordinates free, and distinguish the
location of the unique length-one gap.
\begin{enumerate}
\item If it joins two long runs, also free its zero. Its individual
  derivative vanishes, since setting that zero to one creates exactly
  six sextic monomials. Every interval meeting one of the isolated
  ones is blocked by a fixed zero, and no interval can contain both
  the selected gap zero and an isolated one.
\item If it joins the two isolated short runs, free that intervening
  zero. The three free positions form the local segment
  $1_{\rm free}0_{\rm free}1_{\rm free}$, bordered on either
  side by a gap of length two whose zeros remain fixed. Hence any
  surviving run on this subcube has length at most three, so every
  sextic monomial involving these three variables vanishes.
\item If it joins a long run to an isolated short run, consider
  the \emph{other} gap adjacent to that isolated one, which must
  have length two. If the order is
  $1^{6}\,0_{\rm fixed}\,1_{\rm free}\,00\,R$,
  free the \emph{first} zero of $00$ and fix the second; in the
  reversed order, free the zero nearest the isolated one and fix
  the zero nearest $R$. Here $R$ is the next one-run, long or short.
  The unique length-one gap stays fixed. Consequently the isolated
  one and the selected zero cannot occur together in an all-one
  sextic interval, and the fixed second zero blocks interactions
  with $R$. The other isolated one is separated from this local
  segment by fixed zero(s) or by an intervening length-six one-run
  bordered by fixed zeros. All three free variables are therefore
  compatible.
\end{enumerate}
In every case the subcube is constant, giving three jointly free
coordinates.

All inputs are covered, so $C(F_{6,n})\le n-3$.
The standard inequalities $s\le bs\le C$, combined with the
already established equality $s(F_{6,n})=n-3$, yield
$bs(F_{6,n})=C(F_{6,n})=n-3$.
\end{proof}

\section{Further Questions}
Theorem~\ref{thm:complete-even-bs} determines block sensitivity for
all sufficiently large consecutive even-degree monomial rotation
symmetric Boolean functions, and Corollary~\ref{cor:intrinsic-sensitivity} proves their
sensitivity independently of the companion work. The present paper also determines local certificate
complexity for several structured run families and gives full
quadratic and quartic global classifications. The principal remaining
question is whether the global certificate complexity satisfies
\[
C(F_{d,n})\stackrel{?}{=}bs(F_{d,n})
\]
for all sufficiently large lengths in every fixed even degree.
Other questions include short-length exceptions, nonconsecutive
monomial supports, and the matching--transversal gap of the
opposite-necklace difference hypergraphs.

\section*{Declaration on the Use of Artificial Intelligence}

Generative artificial intelligence tools were used for language editing and
for assistance with exploratory and computational checks during the
preparation of this manuscript.  The authors independently verified the
mathematical statements, proofs, and computations presented in the paper and
take full responsibility for the content.

\appendix
\section{Exact finite-state certificate for the quartic residue-four case}\label{app:finite-state}
The following standalone Python 3 script certifies the finite-state identities in (FS).
Boolean matrices are represented by lists of $64$ integer bitmasks;
all loops are finite and all operations exact. This is a
computer-assisted finite identity verification, not an appeal to the
companion paper's superseded finite-state classification. The mathematical reduction to these identities is proved in the preceding section; the reproduced exact code is an essential part of the verification of the finite identities, whereas Theorem~\ref{thm:even-degree-separation} uses explicit witnesses and the intrinsic sensitivity upper bound of Lemma~\ref{lem:intrinsic-sensitivity-upper}. The subsequent global sensitivity corollary is not an input to that earlier proof.
\begin{scriptsize}
\begin{verbatim}
N=64
A=[0]*N; B=[0]*N
for s in range(N):
 for t in (0,1):
  z=[(s>>i)&1 for i in range(6)]+[t]
  d=(z[4]&z[5]&z[6])^(z[2]&z[4]&z[5])^(z[1]&z[2]&z[4])^(z[0]&z[1]&z[2])
  nxt=(s>>1)|(t<<5)
  (A if d else B)[s] |= 1<<nxt
I=[1<<i for i in range(N)]
def matmul(X,Y):
 out=[]
 for mask in X:
  bits=0
  while mask:
   bit=mask&-mask; bits|=Y[bit.bit_length()-1];mask-=bit
  out.append(bits)
 return out
def power(X,k):
 out=I
 for _ in range(k):out=matmul(out,X)
 return out
def tr(X):return sum(bool(mask&(1<<i)) for i,mask in enumerate(X))
assert power(A,9)==power(A,3)
for d in (1,2,3):
 M=matmul(matmul(matmul(B,power(A,d-1)),B),power(A,9-d))
 assert tr(M)==0,(d,tr(M))
print('verified: A^9 = A^3 and all three forbidden traces vanish')
\end{verbatim}
\end{scriptsize}

\section{Exact verification code for the local lemma and small cases}\label{app:verification}
This complete program requires Python 3.10 or later (because it uses \texttt{int.bit\_count()}). It verifies the local lemma and computes the global maxima of sensitivity, block sensitivity, and certificate complexity for $n=5,9$ by examining every input. It uses no external packages. In the local test, the mixed-derivative condition is evaluated only when all three first derivatives vanish. The $n=9$ check is computationally heavier: it examines $512$ inputs and up to $511$ masks per input, with additional subset searches for certificates and recursive matching for block sensitivity; runtime depends on the machine and Python implementation. The source package also includes an identical standalone copy.
\begin{scriptsize}
\begin{verbatim}
"""Independent exact checks for local three-nonsensitive lemma and n=5,9.
Requires Python 3.10+. No external packages.
Run: python verify_rsbf_small_cases.py
"""
from functools import lru_cache

def quartic(x, n):
    out = 0
    for i in range(n):
        out ^= int(all((x >> ((i+j) % n)) & 1 for j in range(4)))
    return out

def derivative(x, i, n):
    return quartic(x,n) ^ quartic(x ^ (1<<i),n)

def mixed(x, i, j, n):
    return derivative(x,i,n) ^ derivative(x ^ (1<<j),i,n)

# For n>=10, each derivative at 0,1,2,3 depends only on
# x_{-3},...,x_6. We embed this window in a 13-cycle.
for triple in [(0,1,2), (0,1,3)]:
    nonsensitive = all_mixed_one = 0
    for mask in range(1<<10):
        x = sum(1 << (j%13) for t,j in enumerate(range(-3,7)) if (mask>>t)&1)
        if all(derivative(x,i,13)==0 for i in triple):
            nonsensitive += 1
            if all(mixed(x,triple[a],triple[b],13)==1
                   for a,b in [(0,1),(0,2),(1,2)]):
                all_mixed_one += 1
    print('local',triple, 'nonsensitive_assignments=',nonsensitive,
          'all_three_mixed_one=',all_mixed_one)
    assert all_mixed_one == 0

for n, expected in [(5,(4,4,4)), (9,(6,7,7))]:
    N=1<<n
    values=[quartic(x,n) for x in range(N)]
    maxima=[0,0,0]
    for x in range(N):
        opposite=[mask for mask in range(1,N) if values[x^mask] != values[x]]
        sens=sum(values[x^(1<<i)]!=values[x] for i in range(n))
        # Certificate: complement of a maximal free subset containing no opposite mask
        C=min((fixed.bit_count() for fixed in range(N)
               if all(mask&fixed for mask in opposite)),default=0)
        @lru_cache(None)
        def match(available):
            return max((1+match(available^mask) for mask in opposite
                        if mask&available==mask),default=0)
        bs=match(N-1)
        maxima=[max(maxima[i],v) for i,v in enumerate((sens,bs,C))]
    print('global',n, tuple(maxima))
    assert tuple(maxima)==expected
\end{verbatim}
\end{scriptsize}

\end{document}